\documentclass[11pt]{article}
\usepackage[utf8]{inputenc}

\usepackage{amssymb, amsthm, amsmath}
\usepackage{commath, hyperref}
\usepackage{enumerate}
\newtheorem{theorem}{Theorem}

\newtheorem{lemma}{Lemma}
\newtheorem{ass}{Assumption}
\newtheorem{corollary}{Corollary}
\usepackage{setspace}
\usepackage{natbib}
\usepackage{thm-restate}
\usepackage{graphicx}
\usepackage{caption}
\usepackage{subcaption}
\graphicspath{ {./plots/} }

\usepackage{fullpage}
\usepackage{comment}

\newcommand{\ba}{\boldsymbol{a}}

\newcommand{\bc}{\boldsymbol{c}}
\newcommand{\bd}{\boldsymbol{d}}

\newcommand{\bp}{\boldsymbol{p}}

\newcommand{\bs}{\boldsymbol{s}}

\newcommand{\bu}{\boldsymbol{u}}
\newcommand{\bv}{\boldsymbol{v}}

\newcommand{\bx}{\boldsymbol{x}}
\newcommand{\by}{\boldsymbol{y}}
\newcommand{\bz}{\boldsymbol{z}}

\usepackage[linesnumbered, ruled]{algorithm2e}
\allowdisplaybreaks

\usepackage[mathscr]{euscript}
\usepackage{tikz}

\title{Quasi-Newton Acceleration of EM and MM Algorithms via Broyden's Method}

\author{
  Medha Agarwal \\
  Department of Statistics \\
  University of Washington\\
  \texttt{medhaaga@uw.edu} \and
  Jason Xu\\
  Department of Statistical Science\\
  Duke University\\
  \texttt{jason.q.xu@duke.edu}
}

\begin{document}

\maketitle
\doublespacing

\begin{abstract}
The principle of majorization-minimization (MM) provides a general framework for eliciting effective algorithms to solve optimization problems. However, they often suffer from slow convergence, especially in large-scale and  high-dimensional data settings. This has drawn attention to  acceleration schemes designed exclusively for MM algorithms, but  many existing designs are either  problem-specific or rely on approximations and heuristics loosely inspired by the optimization literature. We propose a novel, rigorous  quasi-Newton method for accelerating any valid MM algorithm, cast as seeking a fixed point of the MM \textit{algorithm map}. The method does not require specific information or computation from the objective function or its gradient, and enjoys a limited-memory variant amenable to efficient computation in high-dimensional settings. By connecting our approach to Broyden's classical root-finding methods, we establish convergence guarantees and identify conditions for linear and super-linear convergence. These results are validated numerically and compared to peer methods in a thorough empirical study, showing that it achieves state-of-the-art performance across a diverse range of problems.
\end{abstract}

\section{Introduction}
Iterative procedures are becoming increasingly prevalent for statistical tasks that are cast as optimization of an objective function \citep{everitt2012introduction}. The canonical setting of minimizing a measure of fit together with a penalty term sits at the heart of statistics, yet challenges still arise from high dimensionality, missing data, constraints, and other aspects of contemporary data. 
The principle of majorization-minimization (MM) provides a framework for designing effective algorithms well-suited for such problems lacking a closed form solution. Perhaps the most well-known special case is the expectation-maximization (EM) algorithm, a workhorse for maximum likelihood estimation under missing data. Besides EM, instances of MM abound in statistics, ranging from matrix factorization \citep{lee1999learning} to nonconcave penalized likelihood estimation \citep{zou2008one}. The MM principle is attractive because it admits algorithms that (1) are  simple to implement and (2) provide stable performance by obeying monotonicity in the objective  \citep{dempster1977maximum,laird1978nonparametric}.

Consider $\bx \in \mathbb{R}^p$ and the goal of minimizing a ``difficult" objective function $f : \mathbb{R}^p \to \mathbb{R}$, i.e. finding
$\displaystyle \bx^\ast = \text{argmin}_{\bx} f(\bx)$, which is not available in closed form.  
An MM algorithm transfers this task onto an iterative scheme, successively minimizing a sequence of surrogate functions which dominate the objective function $f$ and are tangent to it at the current iterate  $\bx_k$. This  renders an MM algorithm map $F$, updating $\bx_k$ to $\bx_{k+1} = F(\bx_k)$.

However, MM algorithms typically converge at a locally linear rate, which can translate to impractically slow progress in many statistical problems, especially in high dimensions \citep{wu1983convergence,boyles1983convergence,meng1994global}. To address this issue,  a body of work designs general acceleration schemes for numerical optimization methods including Nesterov's schemes \citep{nesterov1983method}, SAG \citep{schmidt2017minimizing}, SAGA \citep{defazio2014saga}, catalyst acceleration \citep{lin2017generic}, and  SDCA \citep{shalev2014accelerated}. Special attention has also been given towards acceleration methods specifically designed for MM algorithms \citep{jamshidian1997acceleration, jamshidian1993conjugate, lange1995quasi, zhou2011quasi}. Broadly, these methods  seek additional information to better inform the search direction and/or step lengths of the unadorned algorithm. Improvements may come from high order differentials of the objective or  MM algorithm map that incur additional computational cost. As a result, it becomes necessary to balance these tradeoffs.

One approach employs  \textit{hybrid} accelerators  \citep{jamshidian1997acceleration} rely on working directly on the original objective function $f$  \citep{lange1995quasi, jamshidian1997acceleration, jamshidian1993conjugate}, often requiring second order derivatives denoted by $d^2f$. Approximations to $d^2f$ can then be obtained via Fisher scoring in the case of EM. Outside of the context of missing data, classical tools such as quasi-Newton and conjugate gradient methods can be applied to the objective to similar effect. However, an advantage of MM algorithms lies in sidestepping unwieldy objectives in favor of operating on simpler surrogates---for instance, EM works well because it bypasses the need to consider the observed data log-likelihood. Therefore, hybrid methods unfortunately fail to preserve this key advantage of MM algorithms.

An alternative is to instead consider accelerating the MM algorithm map $F$ directly in a way that is largely agnostic to the optimization objective. These have been classified as \textit{pure} accelerators; see \cite{jamshidian1997acceleration}. One class of pure first-order accelerators is given by quasi-Newton (QN) algorithms, which utilize an approximate Jacobian of $F$, denoted by $ d F$, to find a fixed point $\bx^*$ such that $F(\bx^*) = \bx^*$. Equivalently, the goal is to find the root of the MM residual $G(\bx) := F(\bx) - \bx$. \cite{jamshidian1997acceleration} explicitly apply Broyden's classical root-finding algorithm \citep{broyden1965class} for this purpose. As we outlay our method, we will demonstrate that further improvement can be achieved by modifying a general Broyden-type method \citep{broyden1973local} that leverages extra information from the MM map. The STEM and SQUAREM methods of \cite{varadhan2008simple} approximate this Jacobian  by a scalar multiple of the identity matrix. The QN method of \cite{zhou2011quasi} proposes a computationally elegant approximation derived from an assumption of nearness to the stationary point. These pure accelerators tend to preserve the simplicity, convergence properties, and low computational cost of the original algorithm. However, they often rely on heuristic approximations of $dG$ (or $dG^{-1}$), or derivations that potentially ignore a large amount of crucial first-order information. 
While loosely inspired by the theory behind classical quasi-Newton methods, it can be argued that these methods do not fully and formally take advantage of the prior optimization literature.

This paper seeks to fill the methodological gap by proposing a generic accelerator for any MM algorithm map via a quasi-Newton root-finding method. Casting the problem as  root-finding leads to robustness against numerical instabilities. We build off of the wisdom in \cite{zhou2011quasi}, referring to their method as ZAL in this paper, and a few other methods (briefly discussed in Section~\ref{sec:background}) that also seek to find the root of MM residuals using QN method. Various QN methods differ from one another in the way $dG(\bx)^{-1}$ is approximated. Popular methods model this approximation as solution to constrained optimization problem where the linear constraints are provided by \textit{secant approximation} of $G$. Our method incorporates the information from the MM algorithm map to better inform the secant approximation. While ZAL minimizes the norm of the Jacobian near the fixed point, we optimize a richer objective that directly ties into the classical approach of minimizing the change in the Jacobian across iterations, furnishing a rank-two update formula for $dG^{-1}$. This simple yet effective approach guarantees MM acceleration in a more generalized setting and allows us to establish theoretical convergence guarantees. 

This paper is organised as follows: in Section~\ref{sec:background}, we present a general background of MM algorithms and existing acceleration techniques. The key contribution of this paper is formally proposing the MM acceleration algorithm and writing its proof of convergence in Section~\ref{sec:BQN}. Our standard quasi-Newton recipe demands storing the approximate Jacobian matrices in each iteration, which can be computationally ineffective for high dimensions. To address this issue, we further propose a limited-memory variant of our method amenable to high-dimensional settings. We then assess the performance of our algorithm empirically in Section~\ref{sec:examples}, followed by discussion.

\section{Background: EM, MM, and Acceleration}\label{sec:background}

MM algorithms are increasingly popular toward  solving large-scale and high-dimensional optimization problems in statistics and machine learning \citep{lange2008mm, zhou2015novel,xu2019power}.
An MM algorithm minimizes the objective function $f$ by successively minimizing a sequence of surrogate functions $g(\bx \mid \bx_k)$ which dominate the objective function $f(\bx)$ and are tangent to it at the current iterate  $\bx_k$. 
That is, they require  that $g(\bx_k \mid \bx_k) = f(\bx_k)$ and $g(\bx \mid \bx_k)  \geq f(\bx)$ for all $\bx$ at each iteration $k$. Decreasing $g(\bx \mid \bx_k)$  automatically engenders a decrease in $f(\bx)$.   The resulting update $\bx_{k+1}~=~\text{argmin}_{\bx} g(\bx \mid \bx_k)$ implies the string of inequalities $f(\bx_{k+1}) \leq g(\bx_{k+1} \mid \bx_{k}) \leq g(\bx_{k} \mid \bx_{k}) = f(\bx_{k})$
validating the descent property. 
Our method will make use of this practically useful observation.
The MM principle thus converts a hard optimization problem into a sequence of manageable subproblems, expressed as $\bx_{k+1} = F(\bx_k)$.

The iteration terminates when a chosen vector norm (usually $L_2$ norm) of differences between two consecutive iterates is small enough, i.e. $\|\Delta \bx_k\| = \|\bx_{k+1} - \bx_k\| \leq \epsilon$ for some tolerance $\epsilon > 0$. From the perspective of the algorithm map, the MM algorithm amounts to seeking the root  of $G(\bx) := F(\bx) - \bx$. This approach has paved the way for QN acceleration regimes that attempt to well-approximate the inverse of the Jacobian of $G$ at $\bx_k$; see \cite{luenberger1984linear, dennis1996numerical} for a more detailed discussion. Let $dG(\bx)$ be the differential of $G$ evaluated at $\bx$, then $dG(\bx) = (dF(\bx) - I_p)$ where $I_p$ is the $p \times p$ identity matrix. Denoting the approximation to $dG(\bx_k)^{-1}$ by $H_k$, the QN update of $\bx_k$ is given by
\begin{equation}\label{eq:QN_update}
	\bx_{k+1} = \bx_k - H_k G(\bx_k)\,.  
\end{equation}
A QN method is uniquely defined by the way it approximates $dG(\bx_k)^{-1}$. The thread tying different QN methods is the \textit{secant condition}, which states that $H_k$ is the exact inverse Jacobian of a linear function joining $(\bx_k, G(\bx_k))$ and some other point of choice, say $(\by, G(\by))$. That is, the secant constraint mandates that $H_k$ satisfies
\begin{equation}\label{eq:sec}
	\bx_k - \by = H_k (G(\bx_k) - G(\by))\,.
\end{equation}
In the classical Broyden method, $\by$ is taken to be $\bx_{k-1}$. 
For $\bx  \in \mathbb{R}^p$,  $H_k$ is a $p \times p$ matrix and the secant constraint fixes $p$ degrees of freedom. The remaining $p^2 - p$ degrees entail that Eq.\eqref{eq:sec} is underdetermined, satisfied by infinitely many solutions $H_k$. At this juncture, deriving QN methods proceeds by specifying an additional criterion to admit a well-defined procedure. We now survey various popular approaches along this line of thought.

\subsection{Existing MM Acceleration Schemes} 


Perhaps the most transparent and well-studied QN acceleration scheme was proposed by \cite{jamshidian1997acceleration}. Their method directly applies the QN method for root finding by \cite{broyden1965class} to update $H_k$ at any time point $k$. 
Contributions since have noted that this dense matrix update becomes computationally prohibitive in high dimensions typical of contemporary data. The STEM method by \cite{varadhan2008simple} instead provides a simpler approximation of $H_k$ as only a scalar multiple of the identity matrix. Assuming $H_k = \alpha_k I_p$, three variants of STEM entail slightly different inverse Jacobian approximations under 
\begin{equation} \label{eq:squarem_steplengths}
	\alpha_k^{(1)} = \dfrac{\bu_k^T \bv_k}{\bv_k^T \bv_k}, \qquad\qquad \alpha_k^{(2)} = \dfrac{\bu_k^T\bu_k}{\bu_k^T \bv_k}, \qquad\qquad \alpha_k^{(3)} = -\dfrac{ \|\bu_k\|}{\|\bv_k\|}\,,
\end{equation}
where 
$$\bu_k = F(\bx_k) - \bx_k, \quad \text{and} \quad \bv_k \, = \, G(F(\bx_k)) - G(\bx_k) \, = \,F^2(\bx_k) - 2F(\bx_k) + \bx_k.$$ The scalars $\alpha_k$ in \eqref{eq:squarem_steplengths} can be understood as various steplengths for each update rule. An extension to STEM known as SQUAREM was later proposed by the same authors \citep{varadhan2008simple}, using the idea of a ``squared" Cauchy method which may outperform traditional Cauchy methods. 
While SQUAREM outperforms many acceleration methods and is chiefly regarded for its simplicity, the loss of information due to the identity matrix approximation can remain severe, especially in high dimensional cases. 

More recently, \cite{zhou2011quasi} propose an effective acceleration scheme which we will refer to as ZAL. It enjoys the same computational complexity as SQUAREM by avoiding matrix approximation of $dG(\bx_k)^{-1}$ in Eq.\eqref{eq:QN_update} at each step, with update rule
\[
\bx_{k+1} = \bx_k - \left(I_p - \dfrac{\bv_k \bu_k^T}{\bu_k^T \bu_k} \right)^{-1} G(\bx_k) = (1 - c_k)F(\bx_k) + c_k F^2(\bx_k)
\]
where $c_k = \bu_k^T\bu_k/\bu_k^T\bv_k$ and the differences $\bu_k, 
\bv_k$ are as defined earlier. It is worth mentioning the secant constraint used in ZAL, as we will motivate a similar constraint to define the endpoints for our method in the next section. Let $M := dF(\bx^\ast)$. ZAL assumes that $\bx_k$ is close to the optimal point $\bx^\ast$ so that the following linear approximation is reasonable: 
\begin{equation} \label{eq:ZALsecant}
	F \circ F(\bx_k) - F(\bx_k) \approx M(F(\bx_k) - \bx_k)\,.   
\end{equation} 
At each step, ZAL attempts to better approximate $M$ using new information from QN updates in the iteration defined by Eq.~\eqref{eq:QN_update}. 
As these methods operate only with reference to the algorithm map, largely ignoring the objective function to be minimized, they will serve as comparisons for our proposed method. In the following section, we begin by presenting our secant approximation and examine its merits within  Broyden's quasi-Newton paradigm on a univariate illustrative example. 

\section{A novel Broyden quasi-Newton method} \label{sec:BQN}

Recall the QN secant approximation given in Equation~\eqref{eq:sec}. Motivated by \cite{zhou2011quasi}, using the MM update $F(\bx_k)$ as our choice of $\by$ gives the following secant constraint
\begin{equation} \label{eq:bqn_secant_condition}
    F(\bx_k) - \bx_k = H_k(G(F(\bx_k)) - G(\bx_k))\,.
\end{equation}
This serves as our point of departure for deriving our proposed method. Before presenting the details, we begin with an illustrative example that highlights the contrast between the secant condition \eqref{eq:bqn_secant_condition} and  Broyden's standard method.

\paragraph{Illustrative Example.}
We begin by accelerating a classic MM example of minimizing the cosine function  $f(x) = \cos(x)$ to find the root of MM residual.  
To derive a surrogate, consider the following quadratic expansion about $y \in \mathbb{R}$:
\begin{align*} 
	\cos(x) \,& =\, \cos(y)-  \sin(y)(x -y)- \frac{1}{2}\cos(z)(x-y)^2 \\ \,&\, \leq  \cos(y) - \sin(y)(x - y) + \frac{1}{2}(x-y)^2
	\,\, :=\,\, g(x \mid y), 
\end{align*}
where $z$ lies between $x,y$ and the inequality follows since $|\cos(z)|\leq 1$
\citep{lange2016mm}. It is straightforward to minimize $g$ and obtain the nonlinear MM update formula
$x_{k+1} = F(x_k) = x_k + \sin(x_k)\,$ where the interest is in finding the root of $G(x) \,=\, F(x) - x \,=\, \sin(x)$. 

Figure~\ref{fig:secants} presents two consecutive steps of the QN method for finding root of $G$ in one scenario after two initial iterations labeled  $A$ and $B$.. Each plot shows the updates resulting from both secant approximations. We use capital letters to denote a point on the Cartesian plane, and lowercase to denote its corresponding $x$-coordinate. In the left figure, note $a$ and $b$ lie on opposite sides of the root $x^\ast = \pi$, marked by the vertical dashed line. 

\begin{figure}
	\centering
	\includegraphics[width = .85\textwidth]{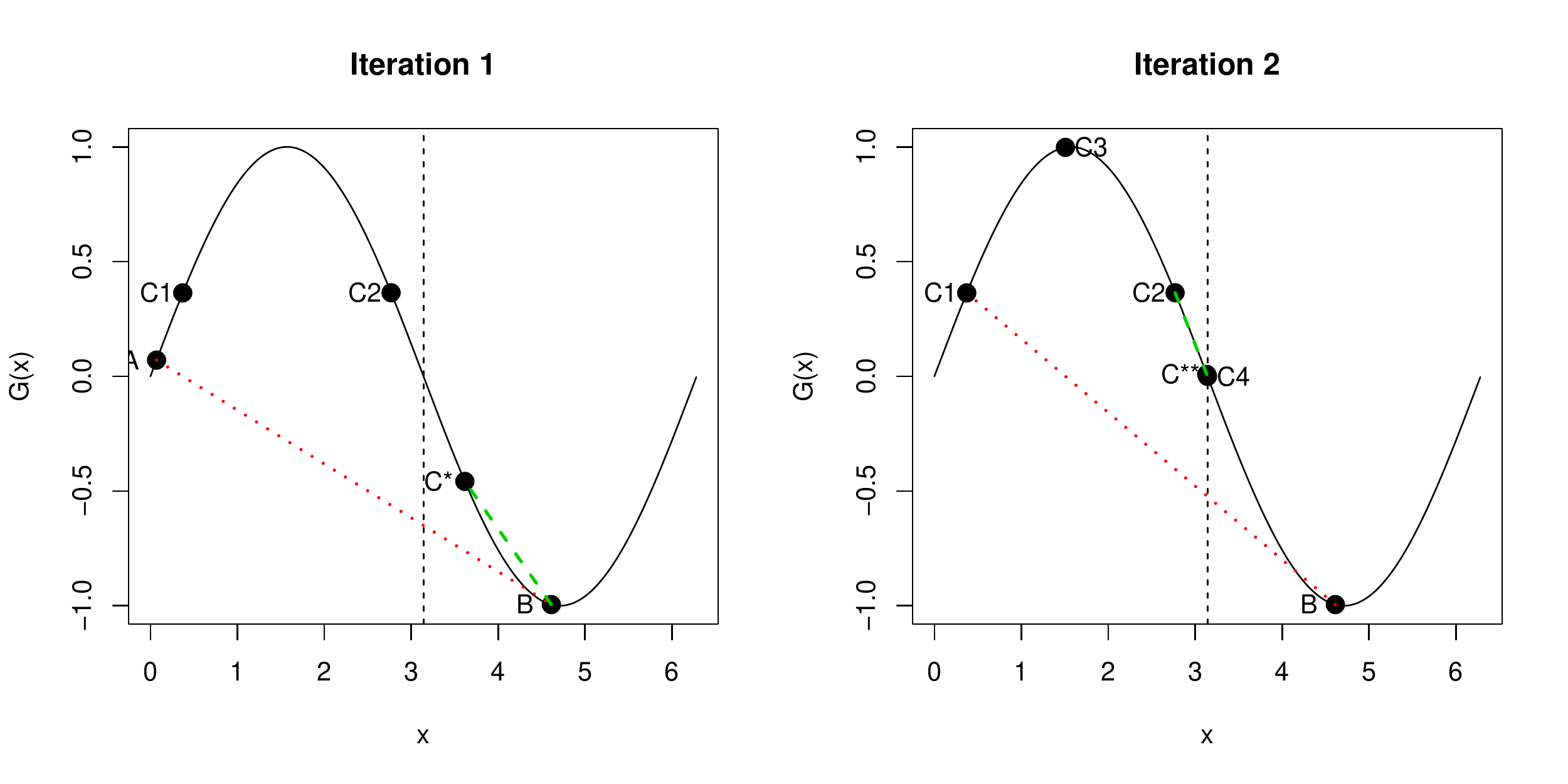}
	\caption{Comparison of secant approximations for two consecutive iterations.}
	\label{fig:secants}
\end{figure}

Now let $c^\ast = F(b)$ be the unaccelerated MM update from $b$.  Rather than moving to $C^\ast$, the standard Broyden iteration approximates the  search direction as the slope of the line joining A and B (dotted red line). Instead, our proposed update derived in the next section employs the line joining $B$ and $C^*$ as the search direction (dashed green line). As a result, Broyden's method produces $C_1$ as the next iterate, and our modified update leads to $C_2$. Specifically, the updates are given by
\[
c_1 = b - G(b)\dfrac{b - a}{G(b) - G(a)} \qquad \text{and} \qquad c_2 = b - G(b)\dfrac{c^* - b}{G(c^*) - G(b)}\,.
\]

The difference in quality of the approximations is revealed upon considering the next step as displayed in the right panel. Let $c^{**} = F(c_2)$, and denote the updates given by standard Broyden's method and our method as $C_3$ and $C_4$, respectively. Not only is it visually evident that standard Broyden iteration is far from the optimum, but our modified Broyden with extrapolation method converges at $C_4$ under an absolute tolerance criterion of $10^{-7}$. 

Even in a univariate setting, where $H_k$ can be completely determined from the secant condition, the advantage provided by our secant approximation is clear when the current state does not render a good search direction.  Because the secants drawn in  standard Broyden's method rely only linearly on the current and previous state, a bad quasi-Newton step propagates to a bad secant approximation, straying far from the fixed point. Our proposed method avoids this by drawing a secant that incorporates information at the current state together with an extrapolation from the next MM step. This additional information  can act as a correction when the original algorithm produces a poor update such as $b$ in the example above.

\begin{table}[h]
	\centering
	\small
	{\begin{tabular}{|c | c| c| c| c| c|}
			\hline
			Method & Minimum & I Quantile & II Quantile & III Quantile & Maximum\\
			\hline
			QN & 2 & 4 & 5 & 6 & 1389\\
			BQN & 1 & 2 & 3 & 3 & 10\\
			\hline
	\end{tabular}}
\caption{Summary of number of iterations until convergence for both methods}
\label{tab:demonstrative}
\end{table}

To showcase the advantage in the illustrative example over a large range of starting points, we run the two QN methods for $1000$ replications, starting from randomly generated points between $0$ and $2\pi$. Table~\ref{tab:demonstrative} gives the summary of the number of iterations until convergence for both methods.

 While this example would have been trivial to optimize directly, it illustrates an advantage that tends to become more pronounced in higher dimensions where the added directional information we harness from the MM extrapolation is richer. Next, we derive our QN acceleration method using the general form of this secant approximation where $p \geq 2$.

\paragraph{Deriving the proposed method.}
Recall that we are interested in finding the root of $G(\bx)$ numerically using QN method. Using the differences $\bu_k, \bv_k$ as introduced in the previous section, the secant condition in Eq.\eqref{eq:bqn_secant_condition} can be expressed as $H_{k}\bv_k = \bu_k$. Note that one may impose several secant approximations $H_k (\bv_k^i) = \bu_k^i$ for $i \in \{1, ..., q\}$ for any $q < p$. These can be generated at the current iterate $\bx_k$ and previous $(q-1)$ iterates, and may yield better performance at the cost of extra computation. To this end, let $U_k = (\bu_1\ \dots \bu_q)$ and $V_k = (\bv_1 \; \bv_2 \dots \bv_q)$ be two $\ p\times q$ matrices; the corresponding linear constraint for $H_k$ in the multiple secant conditions case is $H_k V_k = U_k$.
	
The $p \times p$ inverse Jacobian matrix $H_k$ has $p^2$ degrees of freedom, of which $pq$ degrees of freedom are fixed by the secant approximation. To derive a well-defined update, one must choose how to fix the remaining $p^2 - pq$ degrees of freedom. We follow classical intuitions that yield a connection to Broyden's method for finding roots of nonlinear functions.  
The idea behind this and several of the most successful quasi-Newton methods seeks the smallest perturbation to $H_{k-1}$ when updating to $H_k$, which can also be viewed as imposing a degree of smoothness in the sequence of iterates. The resulting optimization problem can be formulated as
\begin{align} \label{eq:minimization}
	\text{Minimize} &: \|H_{k} - H_{k-1}\|_F\,  \nonumber\\
	\text{subject to} &: H_{k}V_k = U_k\;,
\end{align} 
where $\|\cdot\|_F$ denotes the Frobenius norm. We now take  partial derivatives of the Lagrangian
\[
\mathcal{L} = \dfrac{1}{2}\|H_{k} - H_{k-1}\|^2_F + \Lambda^T(H_{k}V_k - U_k)
\]
with respect to $h_{k}^{ij}$ and set to ${0}$. Here $h_{k}^{ij}$ denotes the $ij^{th}$ element of the matrix $H_{k}$. As a consequence, we obtain the Lagrange multiplier equation
\[
0 = h_{k}^{ij} - h_{k-1}^{ij} + \sum_{k=1}^{p}\lambda_{ik}v_{jk},
\]
which can be expressed in matrix form as
\begin{equation} \label{eq:legrange_eq1}
	H_{k} - H_{k-1} + \Lambda V_k^T = \mathbf{0}   .
\end{equation}
Right-multiplying Eq.\eqref{eq:legrange_eq1} by $V_k$ and imposing the constraint from Eq.\eqref{eq:bqn_secant_condition} gives the solution for $\Lambda$ as
\[
\Lambda = (H_{k-1} V_k - U_k)(V_k^TV_k)^{-1}\,.
\]
Therefore, 
\begin{equation}
	H_{k} = H_{k-1}\left(I_p - V_k (V_k^T V_k)^{-1}V_k^T\right) + U_k(V_k^T V_k)^{-1}V_k^T.
\end{equation}
We remark that as the problem dimension increases, a larger choice of $q$ fixes more information and may improve acceleration, but also risks numerical singularity for the matrix $V_k^T V_k$. We draw attention to the special case of $q=1$ where 
\begin{equation} \label{eq:BFGS_update}
	H_{k} = H_{k-1} - H_{k-1}\dfrac{\bv_k\bv_k^T}{\bv_k^T\bv_k} + \dfrac{\bu_k\bv_k^T}{\bv_k^T\bv_k}\,.
\end{equation}
We see \eqref{eq:BFGS_update} can be written as $H_{k} = H_{k-1} + A_k + B_k$ where both $A_k$ and $B_k$ are rank-1 matrices, yielding a rank-2 update as expected. Also note that the symmetry condition on $H_k$ assumed in classical Broyden-Fletcher-Goldfarb-Shanno (BFGS) updates for minimization is \textit{not} necessary, as here we are approximating an inverse Jacobian  rather than a Hessian/inverse Hessian matrix.

The search direction $\bp_k$ at iteration $k$ is given by $\bp_k = -H_{k}G(\bx_k)$, with a corrected update formula $\bx_{k+1} = \bx_k + \gamma_k \bp_k$, where $\gamma_k = \omega_k /\|\bp_k\|$ is an appropriate scaling factor in the search direction. Here $\omega_k$ is the steplength and $\|\cdot\|$ denotes the $L_2$ vector norm. The corresponding steplength for the unaccelerated MM algorithm is $\|F(\bx_k) - \bx_k\| = \|\bu_k\|$, and for a SQUAREM algorithm is $|\alpha_k^{(i)}| \|\bu_k\|$ for $i \in \{1,2,3\}$. We choose the steplength $|\alpha_k^{(3)}| \|\bu_k\| = \|\bu_k\|^2/\|\bv_k\|$ from \eqref{eq:squarem_steplengths} for our experiments in this paper due to its intuitive explanation \citep{varadhan2008simple}. 
While the behavior of each of the three variants of SQUAREM varies widely, we will see that our method in contrast performs consistently well in a range of scenarios fixing  $\omega_k =\|\bu_k\|^2/\|\bv_k\|$.


\paragraph{Intuition and relation to existing methods.} 

ZAL and SQUAREM are perhaps the most widely used quasi-Newton acceleration methods for MM algorithms. The point of departure in these methods is to cast acceleration as seeking a zero of $G$. 
We ground our approach in the wisdom behind Broyden's root-finding method, and improve upon it by using the secant approximation \eqref{eq:bqn_secant_condition}. As illustrated in the demonstrative example, the benefits of this approximation are twofold. By the descent property of the MM map, $F(\bx_k) - \bx_k$ gives a more reliable direction to move along, 
especially when $\bx_k$ was a  poor update from $\bx_{k-1}$. Second, instead of only one constraint, the MM map enables us to impose multiple linear constraints that become increasingly accurate as iterates $\bx_k$ approaches $\bx^\ast$. 

The STEM and SQUAREM methods employ scalar multiples of the identity as approximations to the Jacobian matrix, which can ignore much valuable curvature information compared to a dense approximation. 
Unlike traditional root finders for non-linear functionals \citep{broyden1965class, pearson1969variable}, their convergence properties are not as rigorously established. The ZAL method makes an assumption that $\bx_k$ is close to the stationary point $\bx^\ast$, validating the linear approximation in Eq.\eqref{eq:ZALsecant}. If $M_k$ denotes the approximation to $dF(\bx^\ast)$ at step $k$, then using the principle of parsimony, the objective in ZAL seeks to minimize $\|M_k\|_F$  subject to the constraint $V_k = (M_k - I_p)U_k$. This criterion yields a computationally elegant update, but unlike Eq.\eqref{eq:minimization} for BQN, effects a disconnect from the theory and intuition behind quasi-Newton methods based on minimally perturbing $H_k$. It is unclear 
how convergence is affected when initiated far from a stationary point where the linear approximation is not reasonably valid. It is our understanding that this approach may fail or converge slowly in such cases, 
since penalizing $M_k$ discourages large steps even when the current estimate is far from stationarity. 

It can be argued that 
a chief advantage of these prior methods is their computational simplicity. In particular, they are quite scalable to high-dimensional problems as their space complexity only grows linearly in the number of variables.
In contrast, while bringing us closer to established optimization theory, our method produces Jacobians that may become computationally unwieldy as the number of variables grows. To ameliorate this,  we next propose a low-memory variant  based on the ideas in limited-memory BFGS. 

\subsection{A limited memory variant for high-dimensional settings} 
Examining Eq.\eqref{eq:BFGS_update} reveals that our algorithm requires the $p \times p$ matrix $H_{k}$ to perform a rank-two update at each step, which can be computationally prohibitive in high dimensions. Additionally, storing the full $p \times p$ matrix at each step can be very challenging. Fortunately, many limited memory variants of  quasi-Newton algorithms have been proposed \citep{shanno1978conjugate,nocedal1980updating,griewank1982partitioned}, and rooting our method in Broyden's  framework allows us to immediately import these ideas. 

We will construct the limited memory version of our algorithm denoted by L-BQN by analogy to the way BFGS algorithm is made scalable using L-BFGS \citep{liu1989limited}. BFGS \citep{fletcher2013practical} is a quasi-Newton optimization method that stores an approximation of the inverse Hessian matrix of the objective function at each iteration. For computationally challenging high-dimensional cases, L-BFGS surpasses this problem by instead storing only a few vectors that represent the inverse Hessian approximation implicitly. Likewise, we will also store a pre-defined $m$ number of vectors that will approximate the inverse Jacobian at each step. Recall that our update is given by $\bx_{k+1} = \bx_k - H_k G(\bx_k)$, where $H_{k}$ is updated by the formula
\[
H_{k+1} \quad=\quad H_{k}\left(I_p - \dfrac{\bv_k \bv_k^T}{\bv_k^T \bv_k}\right) + \dfrac{\bu_k \bv_k^T}{\bv_k^T \bv_k} \quad=\quad H_k W_k +\dfrac{\bu_k \bv_k^T}{\bv_k^T \bv_k},
\]
where $W_k = \left( I- \bv_k \bv_k^T/\bv_k^T \bv_k\right)$. Akin to the L-BFGS method, we may store $m$ previous pairs of $\{\bu_i, \bv_i\}$, $i = k-1, \dots , k-m$, where $m$ typically is chosen between $3$ and $20$. The matrix product required at each step $H_k G(\bx_k)$ can be obtained by performing a sequence of inner products and vector summations involving only $G(\bx_k)$ and the pairs $\{\bu_i, \bv_i\}$, $i = k, \dots, k-m$. After the new iterate is computed, the oldest pair $\{\bu_{(k-m)}, \bv_{(k-m)}\}$ is dropped and replaced by the pair $\{\bu_{k+1}, \bv_{k+1}\}$ obtained from the current step. 

A limited memory variant proceeds by recursion  at each iteration. At the  $k^{th}$ step, an initial estimate of the inverse Jacobian is taken to be a scalar multiple of identity matrix  $H_k^0= \nu_kI_p$. The scale factor $\nu_k$  attempts to capture the size of the true inverse Jacobian matrix along the most recent search direction. 
Next,  $H_k^0$ is updated $(m+1)$ times via Eq.\eqref{eq:BFGS_update} in a nested manner to obtain the relation
\begin{align*}
	H_{k} \; &= \; H_k^0 (W_{k-m} \ldots W_{k}) + \dfrac{\bu_{k-m} \bv_{k-m}^T}{\bv_{k-m}^T \bv_{k-m}}(W_{k-m+1} \ldots W_{k})\\
	\; & \quad + \; \dfrac{\bu_{k-m+1} \bv_{k-m+1}^T}{\bv_{k-m+1}^T \bv_{k-m+1}} (W_{k-m+2} \ldots W_{k}) + \; \ldots
	+ \; \dfrac{\bu_{k} \bv_{k}^T}{\bv_{k}^T \bv_{k }}\,.
\end{align*}
Details on obtaining the nested formula above can be found in Chapter 6 of \cite{nocedal2006numerical}. There the authors suggest that an effective choice for the scaling factor is given by $ \nu_k = {\bu_{k}^T \bv_{k}}/{\bv_{k}^T \bv_{k}}.$
Through this choice, our L-BQN algorithm can be understood as a generalization of the STEM method \citep{varadhan2008simple}: STEM corresponds to the special case where $m=0$. However,  the approximate inverse Jacobian  $\nu_k I_p$ for STEM is derived by minimizing the distance between the zeros of two linear secant-like approximations for $G(\bx)$--- one centered around $\bx_k$, and another at $F(\bx_k)$. While the approaches that lead to this approximation are quite different, it accords more confidence in L-BQN as for non-zero $m$, the inverse Jacobian approximation is made \textit{more} robust by leveraging curvature information from the last $m$ iterates. 

\subsection{Convergence}
We now analyze the convergence properties of the proposed method.
The two essential components 
are \, 1) convergence of the base MM algorithm to the stationary point $\bx^\ast$, and \, 2) convergence of Broyden's root finding quasi-Newton method to the stationary point $\bx^\ast$ of the map $G$. Our study bridges careful analyses of these two facets. 

Naturally, establishing convergence guarantees for our proposed acceleration scheme  rests on the convergence of the underlying MM map, which typically exhibits a locally linear rate of convergence \citep{lange2016mm}. We will assume the base algorithm to be locally convergent in a neighborhood $S$ of $\bx^\ast$ with  rate of convergence denoted by $\tau > 0$. In this section, we prove that BQN is also locally convergent to $\bx^\ast$ in a subset of this neighborhood, and further identify conditions that establish its convergence rate. 
Recall  $\{\bx_k\}$ converges to $\bx^\ast$ at a \textit{linear} rate if, for some chosen vector norm $\|\cdot\|$,
\[
\dfrac{\|\bx_{k+1} - \bx^\ast\|}{\|\bx_k - \bx^\ast\|} \leq r
\]
for some rate of convergence $r \in (0,1)$. The convergence rate is \textit{superlinear} if 
\[
\dfrac{\|\bx_{k+1} - \bx^\ast\|}{\|\bx_k - \bx^\ast\|} \to 0 \qquad\text{ as } k \to \infty\,.
\]

A seminal work of \cite{broyden1973local} derives  local linear and $Q$-superlinear convergence results for several single and double rank quasi-Newton root finding methods. 
Our approach stands close to Broyden's second method, 
while the improved secant approximation through  MM extrapolation will be incorporated into the analysis. We assume that $G$ is differentiable in a neighborhood of $\bx^\ast$, in that the Jacobian matrix $dG(\bx^\ast)$ exists and is non-singular. At many instances, we will treat $(\bx, dG(\bx)^{-1})$ as a tuple whose individual components are updated via Eq.\eqref{eq:QN_update} and \eqref{eq:BFGS_update}. It is crucial to prove that the update function in Eq.\eqref{eq:QN_update} is well defined in some neighborhood of the limit point $(\bx^\ast, dG(\bx^\ast)^{-1})$. To this end, we first prove by induction local convergence of our algorithm under certain conditions. We then carefully construct a neighborhood of $(\bx^\ast, dG(\bx^\ast)^{-1})$ to satisfy these conditions explicitly. 

To ease notation, our current iterate is denoted by $(\bx,H)$ in a neighborhood of $(\bx^\ast, dG(\bx^\ast)^{-1})$. We use $\bar{\bx}$ to denote the update on $\bx$ given by Eq.\eqref{eq:QN_update}, $\bar{H}$ to denote the update on $H$ from Eq.\eqref{eq:BFGS_update}, and introduce further notations: 
$$\bs = \bar{\bx} - \bx,\, \quad \by = G(\bar{\bx}) - G(\bx), \quad  \bu = F(\bx) - \bx, \quad \text{and} \quad \bv = G(F(\bx)) - G(\bx). $$
In the subsequent discussion, suppose $\|\cdot\|$ denotes a chosen vector norm on $\mathbb{R}^p$, then for a $p \times p$ matrix $A$, $\|A\|$ denotes the corresponding induced operator norm. The lemma below supplies useful inequalities to be applied in proving the main theorem.

\begin{lemma} \label{lemma:lipchitz}
	Assume $G: \mathbb{R}^p \to \mathbb{R}^p$ is differentiable in the open convex set $D$, and suppose that for some $\widehat{\bx}$ in $D$ and $d > 0$,
	\begin{equation} \label{eq:lipchitz}
		\|dG(\bx) - dG(\widehat{\bx})\| \leq K \|\bx-\widehat{\bx}\|^d\,,
	\end{equation}
	where $K \in \mathbb{R}$ is a constant. Assuming $dG(\widehat{\bx})$ is invertible, we have for each $\by, \bz $ in $D$,
	\begin{align}
		\|G(\by) - G(\bz) - dG(\widehat{\bx})(\by-\bz)\| & \leq K \max\{\|\by - \widehat{\bx}\|^d, \|\bz - \widehat{\bx}\|^d\}\|\by-\bz\| \label{eq:ineq1}  \nonumber \\
		\|dG(\widehat{\bx})^{-1}(G(\by) - G(\bz)) - (\by - \bz)\| &\leq K \|dG(\widehat{\bx})^{-1}\| \max\{\|\by - \widehat{\bx}\|^d, \|\bz - \widehat{\bx}\|^d\} \|\by-\bz\|\,.
	\end{align}
	Moreover, there exists $\epsilon >0 \text{ and } \rho >0$ such that 
	\[ \max\{\|\by - \widehat{\bx}\|^d, \|\bz - \widehat{\bx}\|^d\} < \epsilon\] 
	implies that $\by \text{ and } \bz$ belong to $D$, and
	\begin{equation} \label{eq:ineq2}
		(1/\rho)\|\by-\bz\| \leq \|G(\by) - G(\bz)\| \leq \rho\|\by-\bz\|\,.
	\end{equation}
\end{lemma}
Inequalities (\ref{eq:ineq1}) follow from standard arguments using Taylor's expansion \citep{ortega2000iterative}, while inequality (\ref{eq:ineq2}) is an immediate consequence of continuity and non-singularity of $dG$ at $\widehat{\bx}$. In the subsequent analysis, we will use a matrix norm $\|\cdot\|_M$, not related to the vector norm $\|\cdot\|$ described earlier. Here, $\|A\|_M := \|MAM\|_F$ where $M$ is a matrix and $\|\cdot\|_F$ is the Frobenius norm. However, there is a constant $\eta > 0$ such that $\|A\| \leq \eta \|A\|_M$ by the equivalence of norms in finite-dimensional vector spaces. 

We now derive general sufficient conditions for local convergence in the spirit of a classic result  by \cite{broyden1973local}. Since we require the inverse of $dG$, we posit the following conditions before proving convergence, with $S \text{ and } D$ as defined earlier.

\begin{ass} \label{ass1}
	(A1) Let the function $G: \mathbb{R}^p \to \mathbb{R}^p$ be differentiable in the open convex set $D$ containing $\bx^\ast$ such that $G(\bx^\ast)=0$ and $dG(\bx^\ast)$ is non-singular. Assume that for some $d > 0$, $G$ satisfies Inequality~\eqref{eq:lipchitz} inside $D$.
\end{ass}

\begin{ass} \label{ass2}
	(A2) 
	Let the update function in Eq.\eqref{eq:QN_update} be well-defined in a neighborhood $N1$ of $\bx^\ast$ where $N_1 \subset D \cap S$, and inverse Jacobian update from Eq.\eqref{eq:BFGS_update} be well-defined in a neighborhood $N_2$ of $dG(\bx^\ast)^{-1}$ containing non-singular matrices. Assume that there are non-negative constants $\alpha_1 \text{ and } \alpha_2$ such that for each tuple $(\bx,H)$ in $N1 \times N2$, the following is satisfied,
	\begin{align} 
		\|\bar{H} - dG(\bx^\ast)^{-1}\|_M &\leq \left[1 + \alpha_1 \max\left\{\|F(\bx) - \bx^\ast\|^d, \|\bx - \bx^\ast\|^d\right\}\right]\|H - dG(\bx^\ast)^{-1}\|_M \nonumber \\
		& \quad + \alpha_2 \max\left\{\|F(\bx) - \bx^\ast\|^d, \|\bx-\bx^\ast\|^d\right\} \label{eq:jacobian_error}\,.
	\end{align}
\end{ass}
The first assumption warrants the application of Lemma~\ref{lemma:lipchitz} on $G$, and the second assumption lends a key error bound on the inverse Jacobian estimation. The notion of well-defined used in Assumption~\ref{ass2} will be qualified for BQN later in Theorem~\ref{th:qnm_convergence}. 

\begin{theorem} \label{th:convergence}
	Let A\ref{ass1} hold true for the function $G$ and A\ref{ass2} be satisfied for some neighborhoods $N_1$ and $N_2$ and non-negative constants $\alpha_1 \text{ and } \alpha_2$. Then for each $r \in (0,1)$ there exist positive constants $\epsilon(r) \text{ and } \delta(r)$ such that   the sequence with $\bx_{k+1} = \bx_k - H_kG(\bx_k)$ is well-defined and converges to $\bx^\ast$ whenever $\|\bx_0 - \bx^\ast\| < \epsilon(r)$ and $\|H_0 - dG(\bx^\ast)^{-1}\|_M < \delta(r)$. Furthermore,
	\[
	\|\bx_{k+1} - \bx^\ast\| \leq r\|\bx_k - \bx^\ast\| \qquad \text{ for each } k \geq 0,
	\]
	and the sequences $\{\|H_k\|\}$ and $\{\|H_{k}^{-1}\|\}$ are uniformly bounded.
\end{theorem}
A detailed proof appears is in the Appendix. 
Under Theorem 1, we inherit the following property by an identical argument of \citet{broyden1973local}, with proof omitted here.

\begin{corollary} \label{cor:superlinear_conv}
	Assume that the conditions of Theorem~\ref{th:convergence} hold. If some subsequence of $\{\|H_k - dG(\bx^\ast)^{-1}\|_M\}$ converges to zero, then $\{\bx_k\}$ converges Q-superlinearly to $\bx^\ast$.
\end{corollary}

It remains to show that our acceleration algorithm satisfies the assumptions of Theorem \ref{th:convergence} and Corollary~\ref{cor:superlinear_conv}. The following result and subsequent corollary identify concrete conditions on the update functions $F \text{ and } G$ that ensure this. 

\begin{theorem} \label{th:qnm_convergence}
	Let A\ref{ass1} hold true for the function $G
	$. If 
	\begin{equation} \label{eq:ineq3}
		\dfrac{\|M\bv - M^{-1}\bv\|}{\|M^{-1}\bv\|} \leq \mu_2\|\bv\|^p,\qquad \bv \neq 0\,,
	\end{equation}
	for a constant $\mu_2 \geq 0$, non-singular and symmetric matrix $M \in \mathbb{R}^{p \times p}$, and all $(\bx, H)$ in a neighborhood $N^\prime$ of $(\bx^\ast, dG(\bx^\ast)^{-1})$, then the update functions \eqref{eq:BFGS_update} is well-defined in a neighborhood $N$ of $(\bx^\ast, dG(\bx^\ast)^{-1})$ and the corresponding iteration
	\[
	\bx_{k+1} = \bx_k - H_kG(\bx_k)
	\]
	is locally convergent to the limit point $\bx^\ast$. 
\end{theorem}
We emphasize that this result does not require stronger conditions than those imposed in the classical results pertaining to Broyden acceleration, 
which have endured as  reasonable mild assumptions in the optimization literature. 
\begin{corollary} \label{cor:q-superlinear}
	If further $\displaystyle \lim_{k \to \infty}\|\bx_{k+1} - F(\bx_k)\|/\|\bx_k - \bx^\ast\| = 0$ holds, then the convergence rate of $\{ \bx_k \} $ to $\bx^\ast$ is Q-superlinear.
\end{corollary}
The complete technical proofs of these results are detailed in the Appendix.
\section{Results and Empirical Performance} \label{sec:examples}

We now turn to a  performance assessment 
on a variety of real and simulated data examples, including (a) quadratic minimization using Landweber's method, (b) maximum likelihood estimation in a truncated beta-binomial model, (c) the largest (and smallest) eigenvalue problem for symmetric matrices, and (d) location-scale estimation of a multivariate $t$-distribution. 
These problems were used in prior studies that introduced the competing methods we benchmark against, thus offering a conservative comparison.
For comparison, we consider (1) unaccelerated MM, (2) the ZAL accelerator, 
(3) the three variants of SQUAREM, and (4) our proposed BQN method as well as (5) its limited memory variant L-BQN.

All methods are implemented using \texttt{R}; we use the implementation of  ZAL and SQUAREM in the \texttt{R} package \texttt{turboEM}.  Throughout our examples, we use the first-order ($K=1$) scheme for SQUAREM as proposed by \cite{varadhan2008simple} as the standard of comparison, since the $K=2$ and $K=3$ schemes are deemed less reliable  by the original authors. The implementation of the proposed accelerators, BQN and L-BQN, and all data examples are implemented as an \texttt{R} package \texttt{quasiNewtonMM} \footnote{\texttt{https://github.com/medhaaga/quasiNewtonMM}}. We consider $q=1$ and $q=2$ secant conditions for the proposed method as well ZAL. 

Stopping criteria are matched across all methods, declaring convergence at $\tilde{\bx}$ when $\|F(\tilde{\bx}) - \tilde{\bx}\| \leq \epsilon$ for a specified tolerance $\epsilon$. For ZAL and BQN, we revert to the original MM step whenever updates violate monotonicity, following \citet{zhou2011quasi}. In most cases,
we observe that BQN performs strikingly well and at least on par with its competitors. An overall theme is that existing methods may outpace our approach on some examples but then falter on a case-by-case basis, while BQN succeeds consistently. 

\subsection{Landweber's method for quadratic minimization}

\begin{table}

	\centering
	\small
	{\begin{tabular}{|c| c c c| } 
			\hline
			Algorithm & $F$ evals  & Time (in sec) & Objective \\ [0.5ex] 
			\hline
			
			MM & 194872.5 (179472.5, 207076.8)  & 4.218 (3.870, 4.470)  & -24.059 (-24.059, -24.059) \\
			BQN, $q=1$ & 5724.0 (4719.5, 6510.0)  & 0.400 (0.330, 0.450) & -24.060 (-24.061, -24.059) \\
			BQN, $q=2$ & 2953.0 (2616.5, 3501.0)   & 0.226 (0.196, 0.274)  & -24.060 (-24.061, -24.059)   \\
			L-BQN & 12856.0 (11260.0, 13772.5)  & 0.631 (0.562, 0.698)  & -24.059 (-24.060, -24.059)   \\
			SqS1 & 2150.0 (1926.5, 2412.0)  & 0.140 (0.127, 0.156) & -24.107 (-24.108, -24.105)\\
			SqS2 & 12665.0 (11515.5, 13855.0)  & 0.833 (0.745, 0.909) & -24.097 (-24.098, -24.097) \\
			SqS3 & 5911.0 (5092.0, 6374.5) & 0.410 (0.348, 0.445)  & -24.106 (-24.106, -24.106)   \\
			ZAL & 23015.5 (21638.2, 24150.7)  & 1.655 (1.544, 1.741) & -24.108 (-24.108, -24.108)  \\[1ex] 
			
			\hline
			
	\end{tabular}}
 \caption{ Quadratic minimization of $f(\theta) = \theta^T A \theta/2 + b^T \theta$ for $100$ random starting points. 
 }
\label{tab:quadratic}
\end{table}

We begin with the ``well-behaved" problem of minimizing a quadratic function $f: \mathbb{R}^p \to \mathbb{R}$ using an MM iterative scheme. For $\theta \in \mathbb{R}^p$, consider a quadratic objective function
\[
f(\theta) = \dfrac{1}{2}\theta^T A \theta + b^T \theta\,,
\]
where $A$ is a $p \times p$ positive definite matrix and $b \in \mathbb{R}^p$. The exact solution is available by solving the linear equation $A\theta = -b$, but incurs a complexity of $\mathcal{O}(p^3)$.
To avoid this computational cost, Landweber's method instead effects an iterative scheme, making use of the Lipschitz property of gradient of $f(\theta)$. The method can be viewed from the lens of majorization-minimization  \citep{lange2016mm}: since $\nabla f(\theta) = A \theta + b$, we can write the gradient inequality
\[
\|\nabla f(\theta) - \nabla f(\Phi)\| = \|A(\theta - \Phi)\| \leq \|A\|\|\theta - \Phi\| .
\]
As a consequence, the spectral norm of A is the Lipschitz constant for $\nabla f(\theta)$. Let the constant $L > \|A\|$. Landweber's method gives the following majorization for $f(\theta)$:
\begin{align*}
	f(\theta) \leq f(\Phi) + \nabla f(\Phi)^T (\theta - \Phi) + \dfrac{L}{2}\|\theta - \Phi\|^2 .
\end{align*}
Minimizing the above surrogate function then yields the MM update formula
\[
\theta_{n+1} = \theta_n - \dfrac{1}{L}\nabla f(\theta_n) = \theta_n - \dfrac{1}{L}(A\theta_n + b) .
\]
Consider the problem dimension to be $p=100$ and tolerance to be $\epsilon = 10^{-5}$. We use a randomly generated $A$ and $b$ such that at true minima, the value of objective function is $-24.10846$. Due to the simple structure of the optimization problem, we might expect all algorithms to perform reasonably well, while we already see the unaccelerated MM algorithm converges very slowly.   
Table~\ref{tab:quadratic} reports performance in terms of the median and interquartile range, comparing the number of $F$ function evaluations ($F$ evals), wall-clock time, and objective values at convergence over $100$ random initializations centered at the true mean, perturbing each component by normal noise with variance $1000$. Figure~\ref{fig:quad_boxplot} displays runtime and function evaluations as boxplots for BQN with $q=1$ (B1), BQN with $q=2$ (B2), L-BQN (L-B), SQUAREM-3, and ZAL. Initial values are matched across methods for each trial.   
Given the strongly convex objective, all methods successfully deliver the minimum here. 

Our proposed BQN method with $q=1$ performs on par with the default SQUAREM-3, while using  $q=2$ secant conditions provides further  improvement. However, we notice that SQUAREM-1 outpaces our method in this case. This may be unsurprising as this ``easy" problem is favorable to methods that do not need to fully utilize curvature information, but are simple and fast. 
It is also worth noting that the variations of SQUAREM already perform quite differently from one another, suggesting significant  sensitivity to the choice of step-length. As mentioned earlier, the performance of ZAL tends to depend on the starting point, and we observe it tends to converge more slowly in this case when initialized with large perturbations of the true value.  

\begin{figure}[htbp]
	\centering
	\includegraphics[width = .9\textwidth]{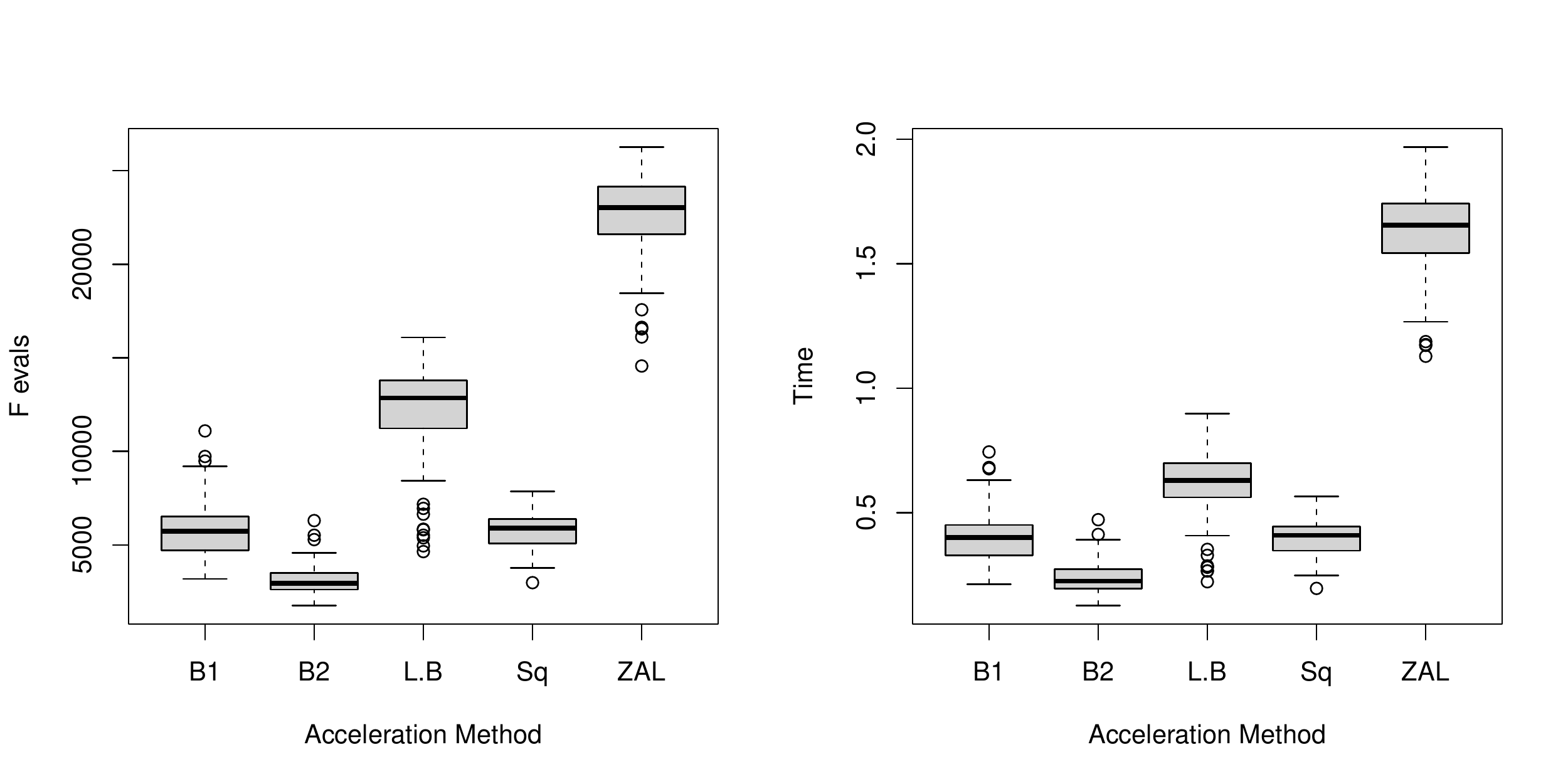}
	\caption{Quadratic minimization: number of function  evaluations and runtime over $100$ random starting points. }

\label{fig:quad_boxplot}
\end{figure}

\subsection{Truncated Beta Binomial} \label{ex:trunc.beta.binom}

We next consider a more difficult statistical optimization problem, turning to the cold incidence dataset by \cite{lidwell1951observations}.  These data have been modeled as a zero-truncated beta binomial model as the reported households have at least one cold incidence. 
The data includes four different household types. We analyze the subset of data corresponding to all adult households here; further details on the data and results for other subsets of the data appear in Table~\ref{tab:beta_binom2} in the Appendix. Among adults, the number of households with $1, 2, 3, \text{ and }4$ cases are $15, 5, 2, \text{ and }2$ respectively.  

Suppose $n$ is the total number of independent observations (households) and $x_i$ denotes the number of cold cases in the $i^{th}$ household. This can be modeled as a discrete probability model \citep{zhou2011quasi} with likelihood given by
\[
L(\theta | X) = \prod_{i=1}^{n}\dfrac{d(x_i| \theta)}{1-d(0|\theta)}\,.
\]
Here $d(x|\theta)$ is the probability density function for a beta binomial distribution with parameter vector $\theta$ and maximum count of $m=4$. Here $\theta = (\alpha, \pi)$ such that $\pi \in (0,1)$ and $\alpha >0$. We use MM algorithm to numerically maximize the likelihood function. The MM updates are given by 

\begin{align*}
	\alpha_{t+1} &= \dfrac{\sum_{j=0}^{m-1} (\dfrac{s_{1j}j \alpha_t}{\pi_t + j\alpha_t} + \dfrac{s_{2j}j \alpha_t}{1 - \pi_t + j\alpha_t})}{\sum_{j=0}^{m-1}\dfrac{r_j j }{1 + j \alpha_t}}\\
	\pi_{t+1} &= \dfrac{\sum_{j=0}^{m-1}\dfrac{s_{1j}\pi_t}{\pi_t + j\alpha_t}}{\sum_{j=0}^{m-1}(\dfrac{s_{1j \pi_t}}{\pi_t + j \alpha_t} + \dfrac{s_{2j}(1-\pi_t)}{1- \pi_t + j \alpha_t})}
\end{align*}
where $s_{1j}, \, s_{2j},\, r_j$ can be interpreted as pseudocounts, given by
\begin{align*}
	s_{1j} &= \sum_{i=1}^{n}1_{x_i \geq j+1}\\
	s_{2j} &= \sum_{i=1}^{n} \left[ 1_{x_i \leq m-j-1} + \dfrac{g(0|\pi_t, \alpha_t)}{1 - g(0 | \pi_t, \alpha_t)} \right]\\
	r_j &= \sum_{i=1}^{n}\left[ 1 +  \dfrac{g(0|\pi_t, \alpha_t)}{1 - g(0 | \pi_t, \alpha_t)} \right] 1_{t \geq j+1}\,.
\end{align*}

\begin{table}[h]
\centering

\begin{tabular}{| c| c c c c|} 
		\hline
		Algorithm & -ln L & $F$ Evals & Iterations & Time (in sec) \\ [0.5ex] 
		\hline
		MM & 25.2283 & 17898 & 17898 & 0.114 \\ 
		BQN ($q=1$) & 25.2287 & 26 & 14 & 0.001 \\
		BQN ($q=2$) & 25.2277 & 29 & 16 & 0.001\\
		L-BQN & 25.2288 & 73 & 37 & 0.002\\
		SqS1 & 25.2274 & 1797 & 1769 & 0.160\\
		SqS2 & 25.2277 & 36 & 19 & 0.004\\
		SqS3 & 25.2269 & 69 & 35 & 0.005\\
		ZAL & 25.2269 & 28 & 24 & 0.003\\[1ex]
		\hline
\end{tabular}
\caption{ Truncated beta binomial: performance on Lidwell and Somerville
		data, from initial point $(\pi,\alpha) = (0.5, 1)$. }
\label{tab:beta_binom1}
\end{table}

\begin{figure}
\centering
\subfloat[MM]{ \includegraphics[width = .45\textwidth]{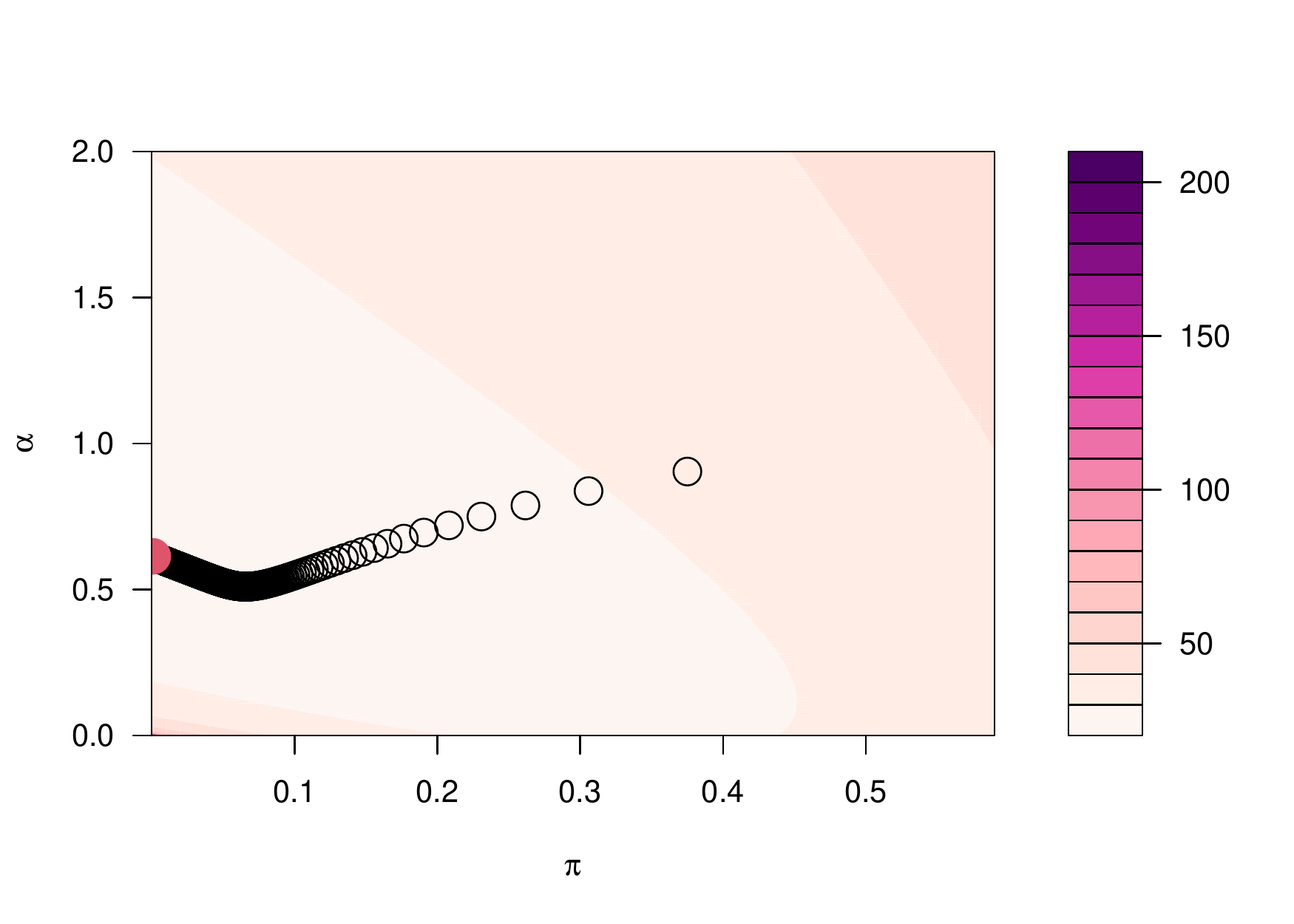}}
\subfloat[BQN, $q=1$]{ \includegraphics[width = .45\textwidth]{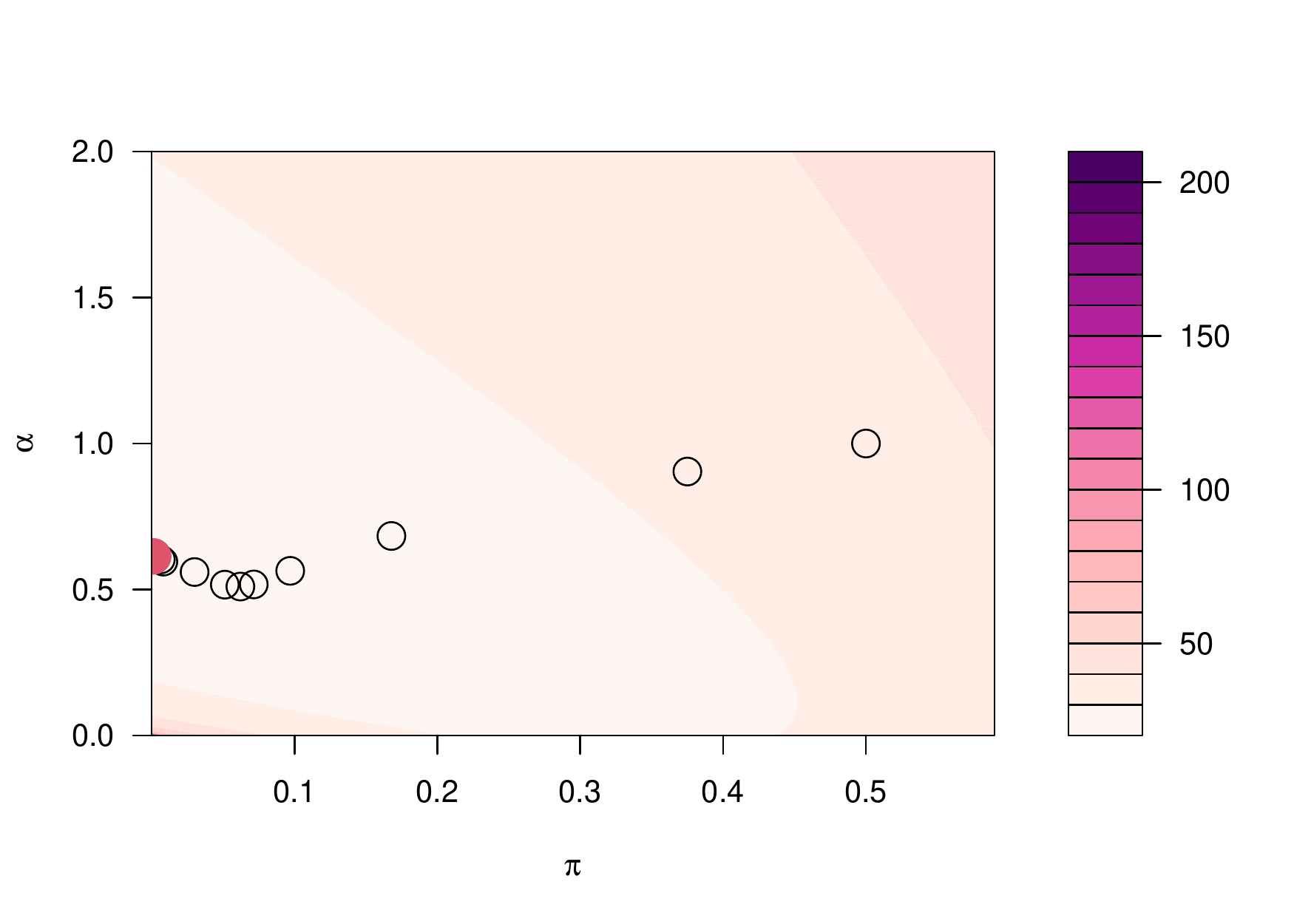}}\\
\subfloat[BQN, $q=2$]{ \includegraphics[width = .45\textwidth]{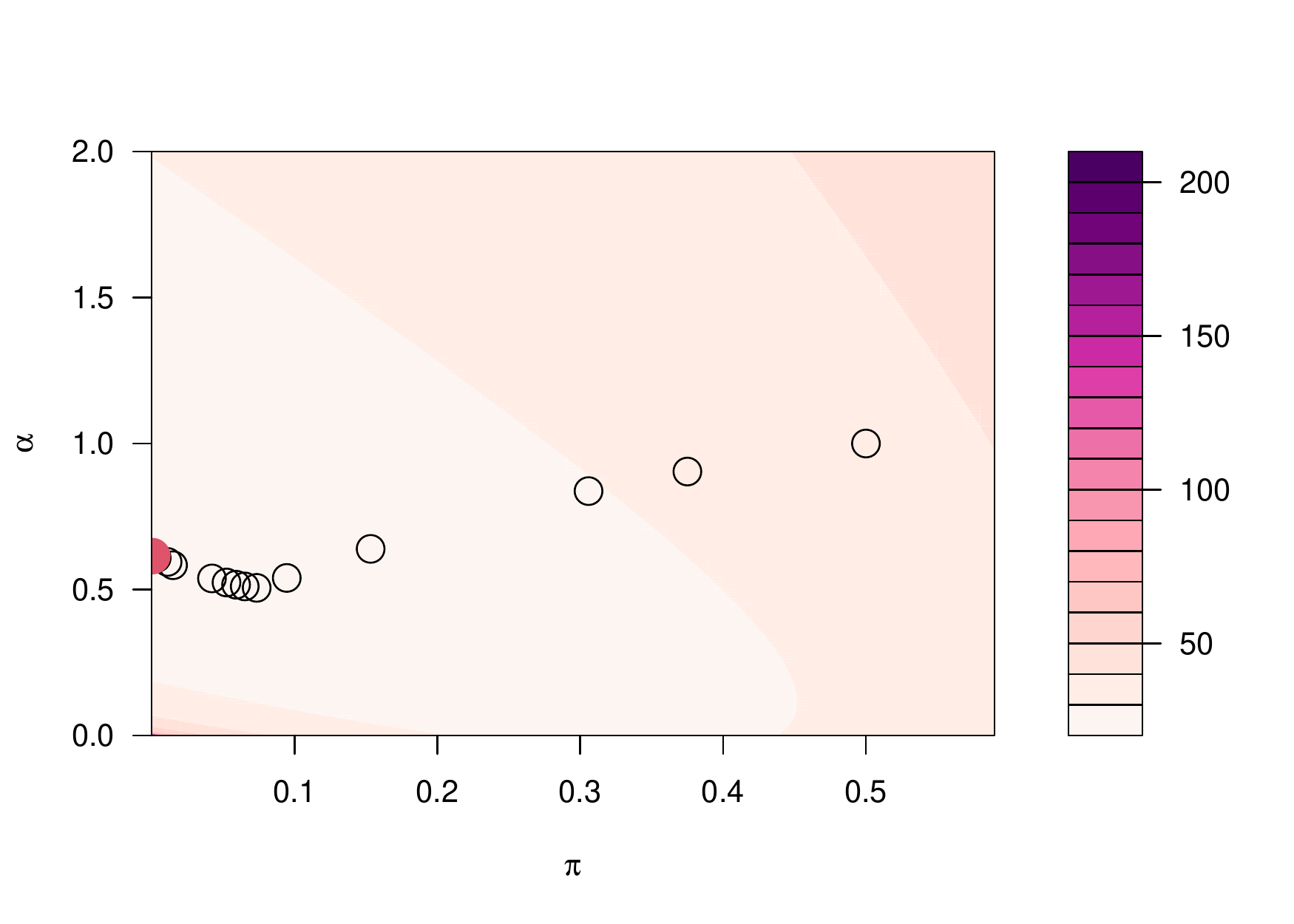}}
\subfloat[L-BQN]{ \includegraphics[width = .45\textwidth]{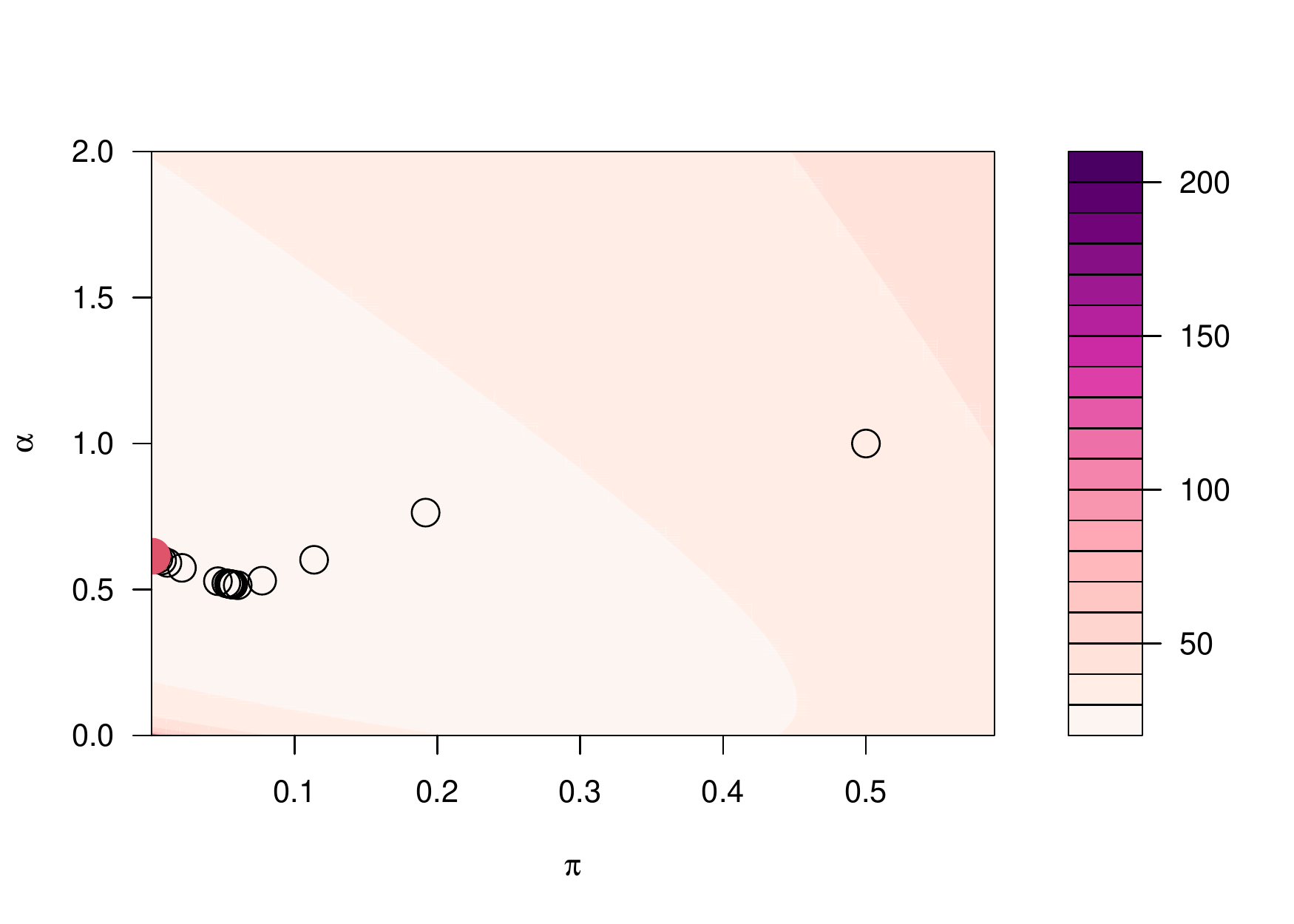}}\\
\subfloat[SQUAREM-3]{ \includegraphics[width = .45\textwidth]{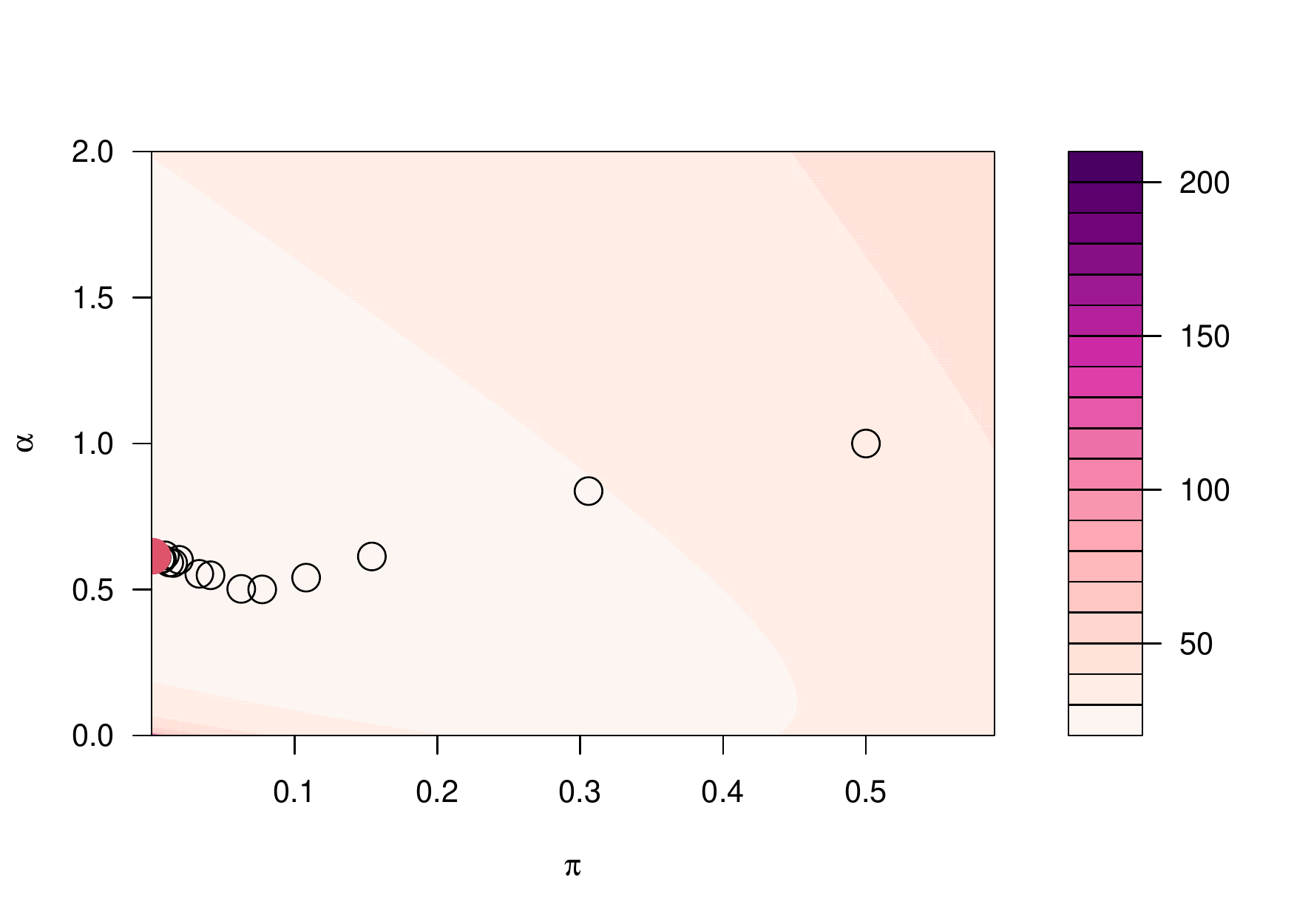}}
\subfloat[ZAL, $q=1$]{ \includegraphics[width = .45\textwidth]{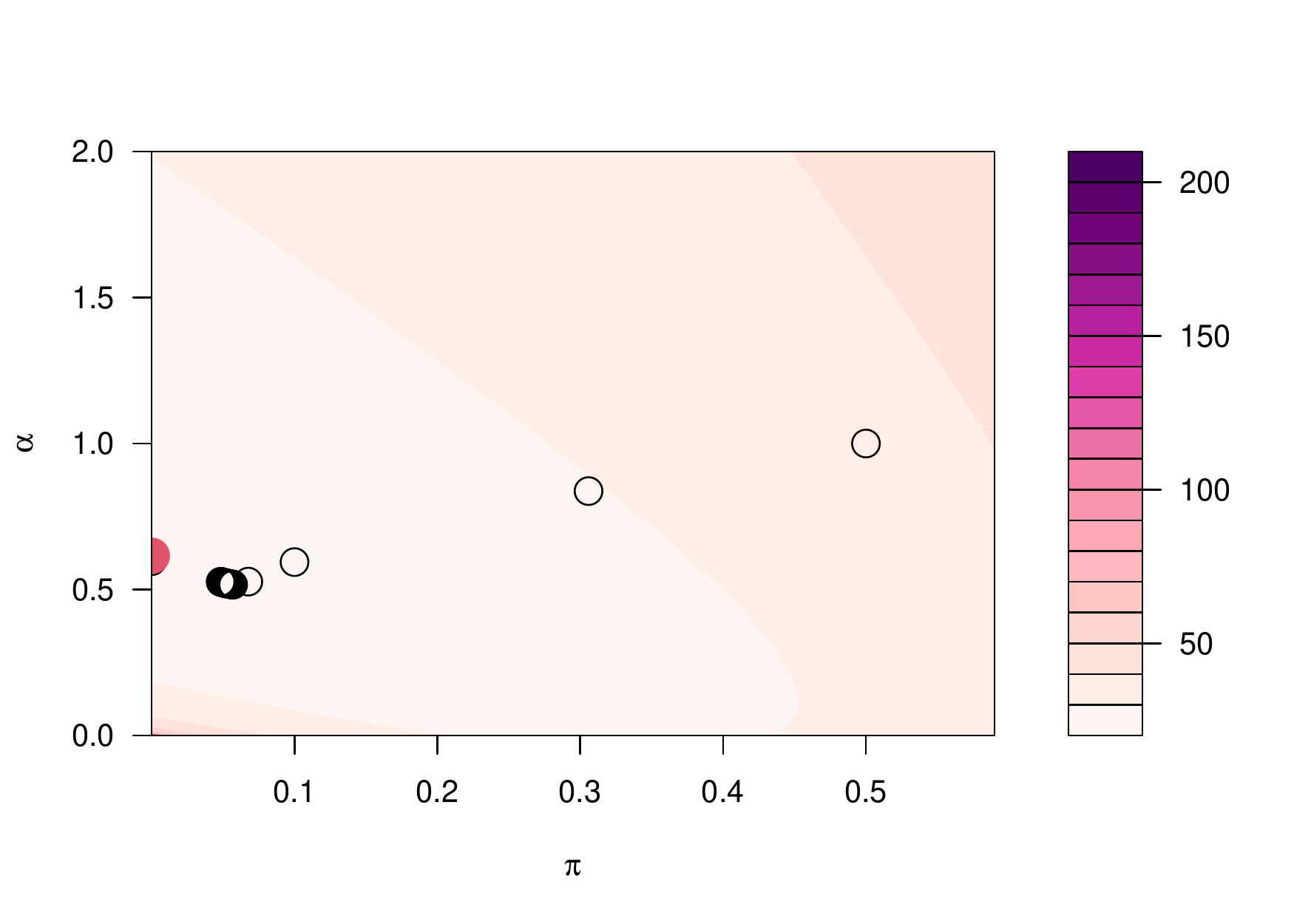}}

\caption{Truncated beta binomial: ascent paths of peer methods on the Lidwell and Somerville household incidence data in a truncated beta binomial model, with  optimum marked in red.}

\label{fig:beta-contour}
\end{figure}

Following \cite{zhou2011quasi}, each algorithm is  initializated at $(\pi, \alpha) = (0.5, 1)$. 
Table \ref{tab:beta_binom1} lists the negative log-likelihood values, number of MM evaluations (F evals), number of algorithm iterations, and runtime until convergence for each algorithm. 
Figure~\ref{fig:beta-contour} provides a closer look, showing the progress path of each algorithm on a contour plot of the objective. SQUAREM methods, though achieving significant acceleration, tend to exhibit slow tail behavior near the optimal value. In particular, SQUAREM-1 leads to orders of magnitude slower convergence than the others, while it outpaced other choices of steplength among variants of SQUAREM in the simple example; we again focus on visualizing the progress of the default SQUAREM-3. In all cases, our method converges in fewer iterations and requires fewer function evaluations than its competitors despite a naive implementation. From Figure~\ref{fig:beta-contour}, we can visualize the advantage of our extrapolation-based steps making steady progress, in contrast to the more congested updates near the optimum under existing methods. While the small problem dimension does not call for a limited-memory method, we see L-BQN also compares favorably despite its streamlined updates. 

\subsection{Generalized eigenvalues} \label{ex:gen.eigen}

In this example, we consider a more complicated objective function that exhibits a zig-zag descent path under the na\"ive MM algorithm, rendering progress excruciatingly slow. For two $p \times p$ matrices $A$ and $B$, the generalized eigenvalue problem refers to finding a scalar $\lambda$ and a nontrivial vector $x$ such that $Ax = \lambda Bx$. We consider the case where $A$ is symmetric and $B$ is symmetric and positive definite, so that the generalized eigenvalues and eigenvectors are real 
\citep{zhou2011quasi}. A simple alternative for finding the generalized eigenvalues iteratively is by optimizing the Rayleigh quotient 
\[
R(x) = \dfrac{x^T A x}{x^T B x} \qquad \qquad x \neq 0\,.
\]
The gradient of $R(x)$ is given by
\[
\nabla R(x) = \dfrac{2}{x^T B x}[Ax - R(x)Bx]\,.
\]

Therefore, a solution of $\nabla R(x)=0$ corresponds to a generalized eigenpair, wherein the maximum of $R(x)$ gives the maximum generalized eigenvalue and minimum gives the minimum generalized eigenvalue. 
To optimize $R(x)$, we consider the line search method for steepest ascent proposed by \cite{hestenes1951solutions} as the base algorithm.

\begin{figure}[!h]
\centering
\subfloat[Smallest Eigenvalue]{\includegraphics[width = .48\textwidth]{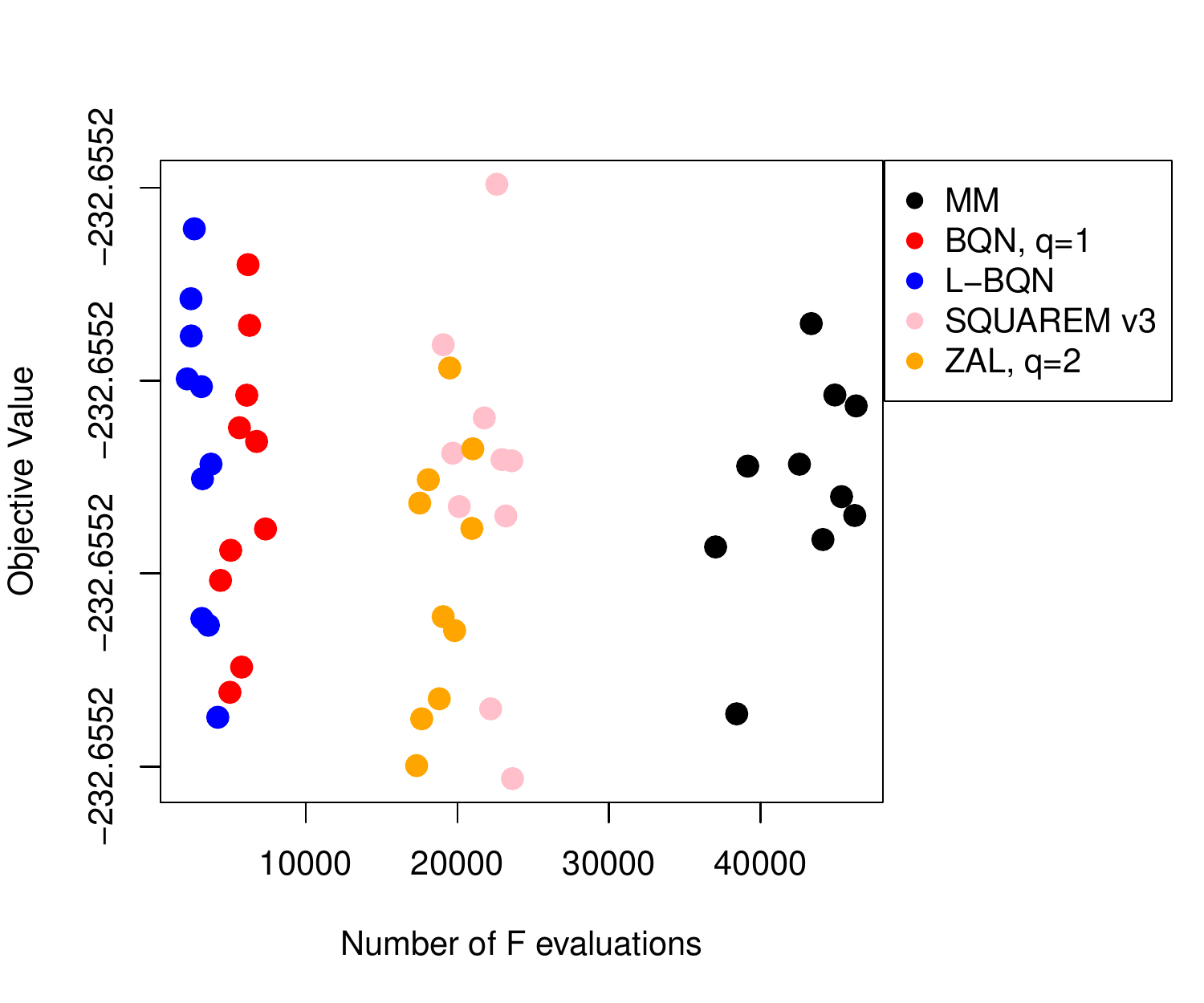}}
\subfloat[Smallest Eigenvalue]{\includegraphics[width = .48\textwidth]{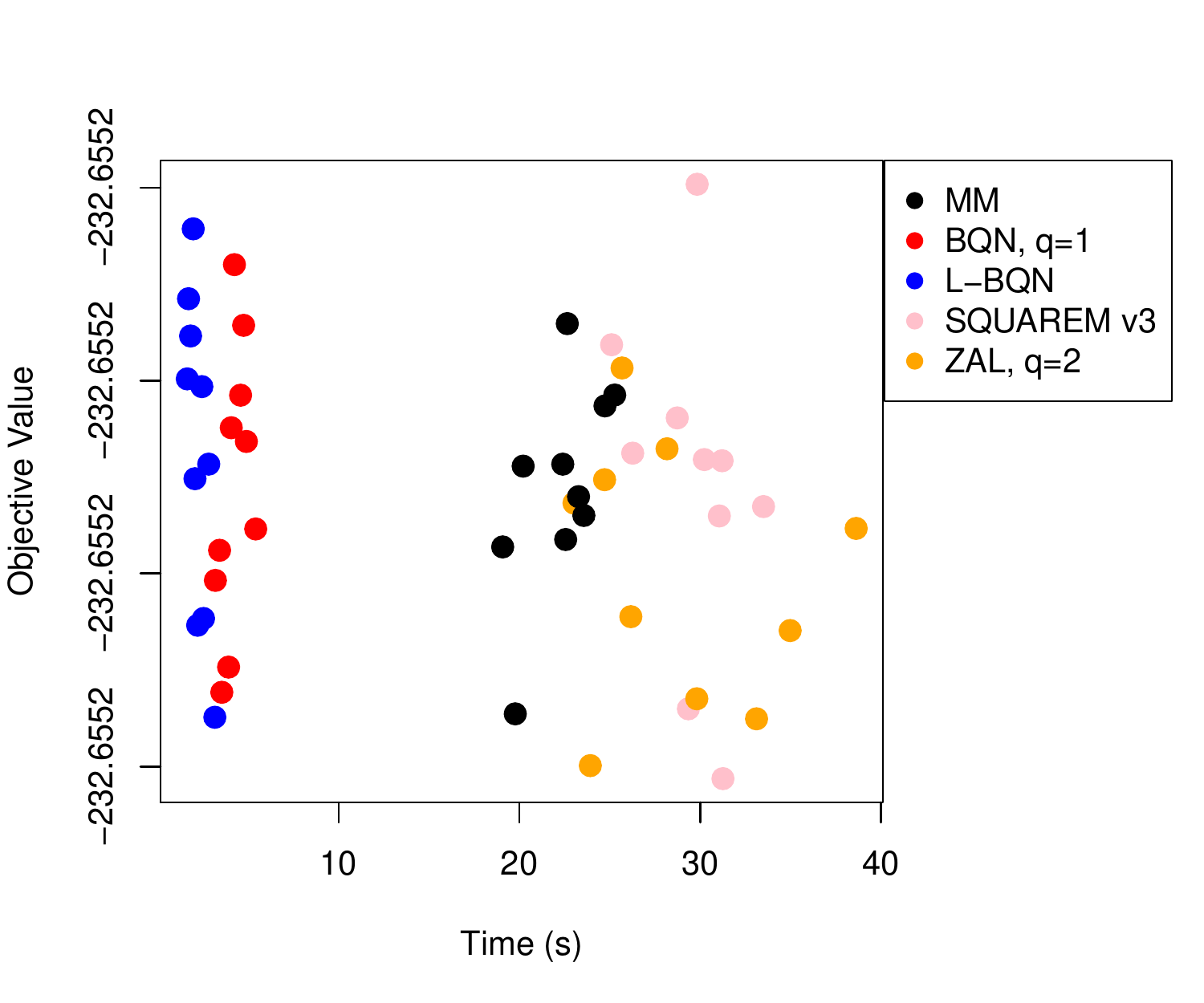}}\\

\subfloat[Largest Eigenvalue]{\includegraphics[width = .48\textwidth]{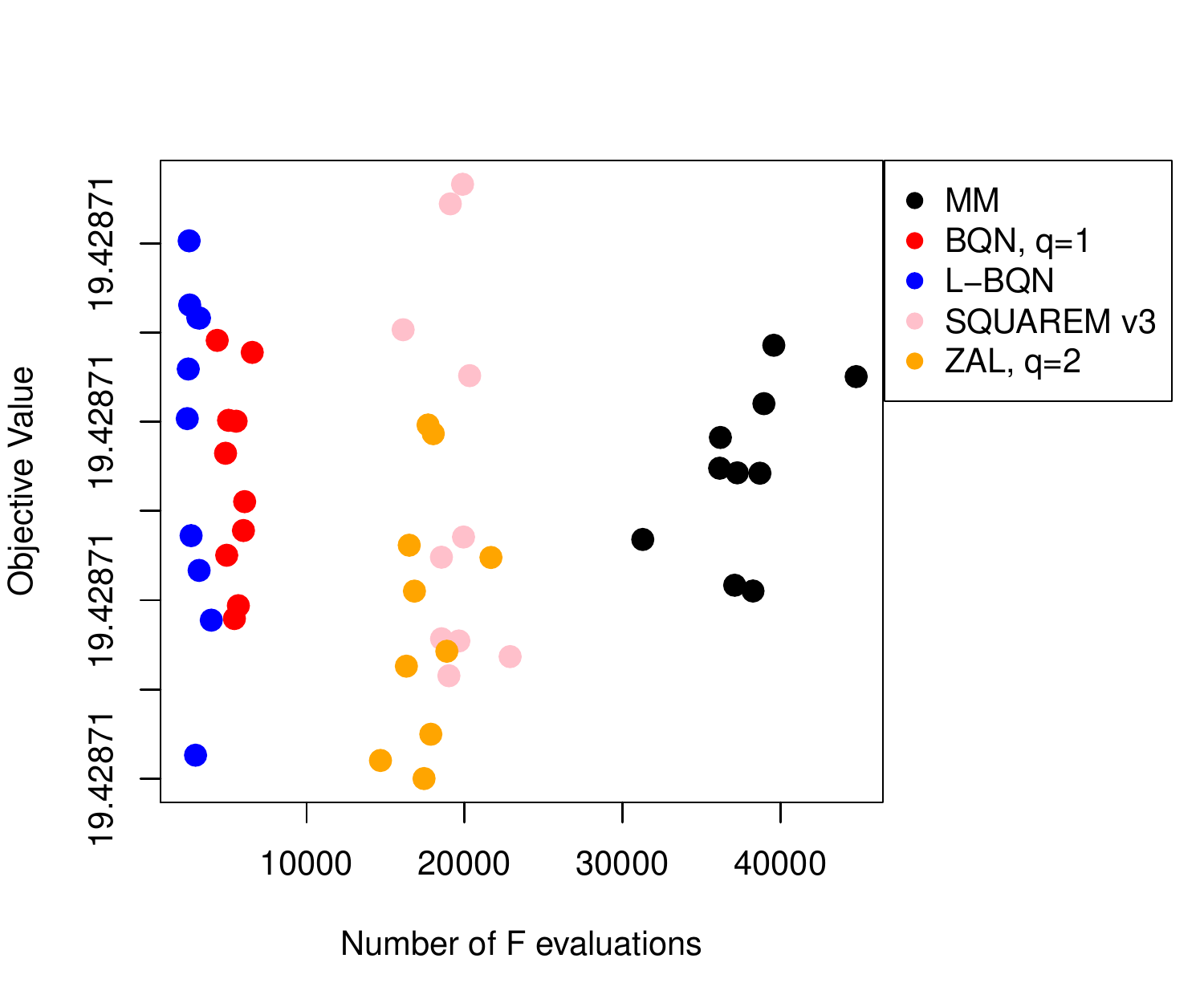}}
\subfloat[Largest Eigenvalue]{\includegraphics[width = .48\textwidth]{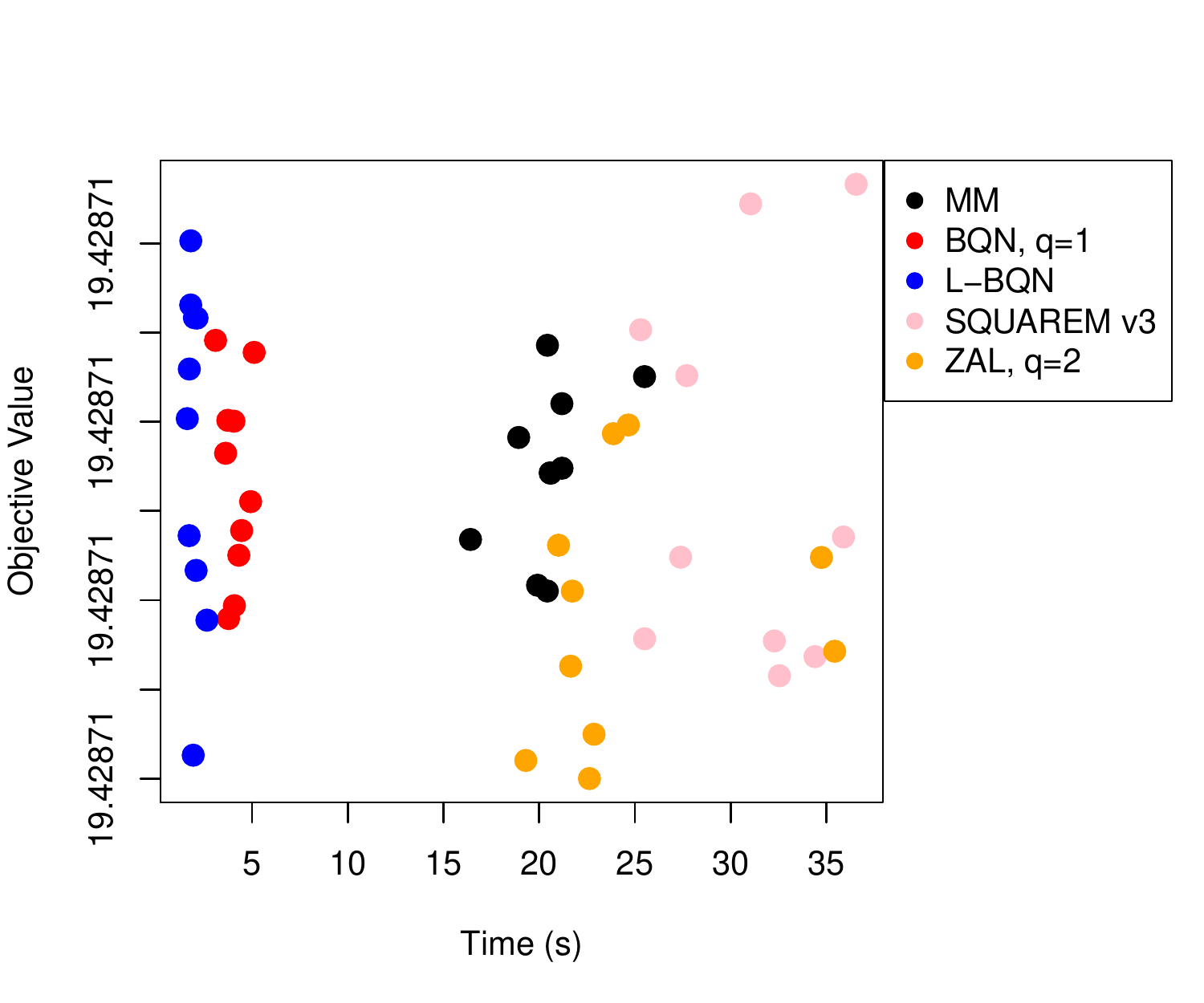}}

\caption{Generalized eigenvalues: we see all methods reach the same objective at convergence up to tolerance, plotted over $10$ restarts against the time and number of F evaluations. }
\label{fig:eigen-sp}
\end{figure}

\begin{table}[h]

\centering

\begin{tabular}{|c | c c c | c c c|} 
\hline
& \multicolumn{3}{|c|}{Smallest Eigenvalue} & \multicolumn{3}{|c|}{Largest Eigenvalue} \\ [0.5ex] 
Algorithm & Time (in sec) & $F$ Evals & Eigenvalue & Time (in sec) & $F$ Evals  & Eigenvalue\\ [0.5ex]
\hline
MM  & 10.104 & 42552 & -232.655 & 8.539 & 38262 & 19.429\\ 
BQN, $q=1$ & 2.046 & 6682 & -232.655 & 1.854 & 5766 & 19.429\\
L-BQN & 1.203 & 4046 & -232.655 & 0.664 & 2399 & 19.429\\
SqS1 & 11.850  & 21777 & -232.655 & 11.397 & 19642 & 19.429 \\
SqS2 & 12.391 & 21777 & -232.655 & 11.003 & 19642 & 19.429 \\
SqS3 & 12.525 & 21777 & -232.655 & 11.021 & 19642 & 19.429\\
ZAL & 9.979 & 17920 & -232.655 & 9.665 & 17415 & 19.429\\ [1ex] 
\hline
\end{tabular}
\caption{
Generalized eigenvalues: number of $F(x)$ evaluations, runtime, and eigenvalues at convergence.}
\label{tab:gen_eigen}
\end{table}

Due to the zigzag nature of steepest ascent on this problem, \cite{zhou2011quasi} found na\"ive acceleration to perform poorly. Utilizing this side information, they considered instead the $s$-fold functional composition of the base algorithm for $s$ even as the underlying map, improving performance. We refrain from using the same heuristic in order to illustrate the  off-the-shelf applicability of our method. 
We consider a simulation study with symmetric matrices $A$ and $B$ randomly generated with $p=100$ dimensions, and run $10$ random initializations of each method from matched initial points.  

Figure~\ref{fig:eigen-sp} displays objective values at convergence, and Table \ref{tab:gen_eigen} details the results. It can be seen that without the $s$-fold functional composition, both SQUAREM and ZAL fails to accelerate meaningfully here. On the other hand, the curvature information is crucial toward informing a good search direction in such cases, and our formulation successfully leverages this information. 
This information is largely ignored in the scalar-based methods SQUAREM, while ZAL attempts to make use of curvature information under an assumption that it is close to the stationary point.


\subsection{Multivariate t-distribution} \label{ex:multi.t.distr}

Our last example turns to estimation under a multivariate $t$-distribution, a robust alternative to  multivariate normal modeling when the errors involve heavy tails \citep{lange1989robust}.  
\cite{varadhan2008simple} considered this example to compare SQUAREM to standard EM as well as PX-EM, an efficient data augmentation method  \citep{meng1997algorithm}. 

Suppose we have $p$-dimensional data $Y = (y_1, ..., y_N)$ that we wish to fit to a multivariate $t$-distribution with unknown degrees of freedom $\nu$. The density is given by 
\[
f(y| \mu, \Sigma) \propto | \Sigma |^{-1/2} \left(\nu + (y - \mu)^T \Sigma^{-1} (y - \mu)\right)^{(\nu + p)/2},
\]
and so the data likelihood is givein by $\prod_{i=1}^{N}f(y_i | \mu, \Sigma)$. There is no closed form solution to find $(\mu, \Sigma)$ which maximize the likelihood, but we can make progress by augmenting the missing data with latent variables. That is, we obtain the complete data $\{(y_i, q_i); i = 1, ..., N\}$ where q are IID from $\chi^2_\nu / \nu$;  the maximum likelihood estimator (MLE) now follows from weighted least squares. In an EM algorithm, the E-step finds the expected complete data log-likelihood conditional on parameters from the previous iteration $k$. 
Conditional on $Y$ and $(\mu_k, \Sigma_k)$, the latent variables are distributed as $q_i \sim \chi^2_{\nu +p}/(\nu + d_i^{(k)})$, where $d_i^{(k)} = (y_i - \mu_k)^T \Sigma_k^{-1} (y_i - \mu_k); i= 1,..., N$. As the complete-data log-likelihood is linear in $q_i$, the E-step amounts to defining
\[
w_i = E[q_i | y_i, \mu_k, \Sigma_k] = (\nu + p)/(\nu + d_i^{(k)}); \qquad i = 1, ..., N\,.
\]
The M-step then yields:
\begin{align*}
\mu_{k+1} &=\sum_{i}w_i y_i \bigg / \sum_{i}w_i\,, \qquad \Sigma_{k+1} \,=\, \dfrac{1}{N} w_i (y_i - \mu)(y_i - \mu)^T\,.
\end{align*}

\begin{table}[h]

\centering
\begin{tabular}{|c | c c c|} 
\hline
Algorithm & $F$ Evals & Time (in sec)  & $-\ln L$ \\ [0.5ex] 
\hline
EM & 744 & 1.086 & 8608.99\\ 
PX-EM & 38 &  0.063 & 8608.99\\
BQN, $q=1$ & 112 & 0.215 & 8608.99\\
BQN, $q=2$ & 223 & 0.457 &  8608.99\\
L-BQN & 89 & 0.114 & 8608.99\\
SqS1 & 64 & 0.156 & 8608.99\\
SqS2 & 65 & 0.148 & 8608.99\\
SqS3 & 63 & 0.155 & 8608.99\\
ZAL, $q = 2$ &  383 & 1.383 &  8608.99\\
[1ex] 
\hline
\end{tabular}
\caption{
Multivariate t-distribution: maximum likelihood estimation of a 25-dimensional multivariate t-distribution. } 
\label{tab:t-dist}
\end{table}

The PX-EM method of \cite{meng1997algorithm} differs only in the $\Sigma$ update, replacing the denominator N by $\sum_{i}w_i$. We randomly generate synthetic data with $\nu = 1$ (a multivariate Cauchy distribution) and parameters $\mu = 0$, $\Sigma = V$, where $V$ is a symmetric randomly generated matrix with dimension $p = 25$, which corresponds to $350$ parameters (25 for $\mu$ and 325 for $\Sigma$). 
We report results obtained from following the initial values suggested by \cite{meng1997algorithm}:

\begin{equation*}
\mu_0 =  \dfrac{1}{N}\sum_{i=1}^{N}y_i, \qquad
\Sigma_0 = \dfrac{1}{N}\sum_{i=1}^{N}(y_i - \overline{y})(y_i - \overline{y})^T\,.
\end{equation*}

Table~\ref{tab:t-dist} displays runtime, number of $F$ evaluations (F evals), and negative log likelihood of all acceleration schemes at convergence. Our method achieves significant acceleration compared to the standard EM algorithm. However, it does not compare well with SQUAREM in this high-dimensional setting. Note that L-BQN performs on par with SQUAREM. Here ZAL fails to provide meaningful acceleration under its implementation in \texttt{turboEM}---we observe it frequently proposes an update such that $\Sigma_k$ is not positive-definite. In these cases, the algorithm reverts to the default MM update, adding additional computational effort, though the implementation in \cite{zhou2011quasi} achieves more success. Though performance is always quite dependent on implementations, we echo the overall theme in the findings of \cite{varadhan2008simple,zhou2011quasi} that  model-specific augmentation under PX-EM performs remarkably well, outpacing all of the more general methods. This example illustrates that despite the robust performance of our proposed method across settings, it is worthwhile to exploit problem-specific structure as does PX-EM whenever possible.

\section{Conclusion}

This article presents a novel quasi-Newton acceleration of MM algorithms that extends recent ideas, but lends them new intuition as well as theoretical guarantees. 
The method  retains gradient information across all components, which is often ignored in other \textit{pure} MM accelerators. 
A key advantage of MM algorithms is their transfer of difficulty away from the original objective function, obtained by the construction of surrogates. While the \textit{hybrid} quasi-Newton MM accelerators \citep{lange1995quasi, heiser1995convergent, lange2000optimization} are rigorously analyzed in the literature, they lose this appeal in part by requiring information from the original objective through their iterates. Our approach seeks to embody the best of both worlds,  retaining the simplicity of pure accelerators without restrictive assumptions, maintaining computational tractability so that it is amenable for large and high-dimensional problems, and taking advantage of richer curvature information that yields classical convergence guarantees which may not hold for its peer methods.

The limited-memory version of our method performs well on our representative, but not exhaustive, set of examples. As this  shows promise toward high-dimensional problems, a fruitful line of research may seek to study the convergence properties of L-BQN explicitly, building on prior analyses on  convergence of limited memory BFGS method \citep{liu1989limited}. Exploring optimal step size selection presents another open direction \citep{nocedal2006numerical}. Despite deriving from a different perspective, it is satisfying that the steplength for our inverse Jacobian update in Eq.\eqref{eq:BFGS_update} reveals that used for the first version of STEM as a special case \cite{varadhan2008simple}. Nonetheless, exploring the practical and theoretical merits of alternatives may reap further advantages.

\bibliographystyle{apalike}
\bibliography{ref}

\section{Appendix} \label{sec:appendix}

\subsection{Proof of Theorem~\ref{th:convergence}}
Let $\mathbb{R}^{p\times p}$ denote linear space of real matrices of order $p \times p$. Recall that $\|A\|_M := \|MAM\|_F$ is a matrix norm of matrix $A$ for any matrix $M$, $\|\cdot\|_F$ is the Frobenius norm, and $\|\cdot\|$ denotes a vector norm or its induced operator norm. 
\begin{proof}[]
	The proof argues that if $\bx_0 \in D$, then $\bx_1$ also lies in D using the inequality in Eq.\eqref{eq:jacobian_error}. Additionally, it is shown that the distance of $\bx_1$ from $\bx^\ast$ is less than or equal to some $r$th fraction of the distance of $\bx_0$ from $\bx^\ast$. By induction, we prove that $\bx_i \in D$ for all $i \geq 1$, and eventually converge to $\bx^\ast$ with $r$ rate of convergence.
	
	To this end, we upper bound the norm of the Jacobian and inverse Jacobian matrices at $\bx^\ast$ as $ \|dG(\bx^\ast)\| \leq \sigma \text{ and } \|dG(\bx^\ast)^{-1}\| \leq \gamma$. For any $r \in (0,1)$, we can choose $\epsilon(r) = \epsilon \text{ and } \delta(r) = \delta$ such that
	\begin{align}
		[2\alpha_1\delta + \alpha_2] \dfrac{\epsilon^d}{1 - r^d} &\leq \delta \label{eq:epsilon_delta1}\\ 
		2\sigma \delta \eta + (\gamma + 2\eta \delta)K \epsilon^d &\leq r \label{eq:epsilon_delta2} \,.
	\end{align}
	If necessary, we may further restrict $\epsilon$ and $\delta$ such that $(\bx,H) \in N$ whenever $\|\bx - \bx^\ast\| < \epsilon$ and $\|H -dG(\bx^\ast)^{-1}\|_M~<~2\delta$. Now suppose $\|\bx_0 - \bx^\ast\| < \epsilon \text{ and } \|H_0 - dG(\bx^\ast)^{-1}\|_M < \delta$. Then  $\|H_0 - dG(\bx^\ast)^{-1}\| < \eta \delta < 2\eta \delta$ by the equivalence of norms in finite-dimensional vector spaces. From Eq.\eqref{eq:epsilon_delta2}, $2 \sigma \delta \eta \leq 2r$, and therefore the Banach lemma gives,
	\[
	\|H_0^{-1}\| \leq \dfrac{\sigma}{1-r}\,.
	\]
	Now, we will show that if $\bx_0 \in D$, then $\bx_1$ also lies in $D$. For this purpose, we add and subtract $H_0 dG(\bx^\ast)(\bx_0 - \bx^\ast)$ and add the null term $H_0G(\bx^\ast)$ to the known update formulation for $\bx_1$ giving
	\begin{align*}
		\bx_1 - \bx^\ast &= \bx_0 - H_0G(\bx_0) - \bx^\ast \\
		&= -H_0\left[G(\bx_0) - G(\bx^\ast) - dG(\bx^\ast)(\bx_0 - \bx^\ast)\right] + \left[I_p - H_0dG(\bx^\ast)\right](\bx_0 - \bx^\ast)\,.
	\end{align*}
	Using the fact that $\|I_p - H_0dG(\bx^\ast)\| = \|dG(\bx^\ast)(H_0 - dG(\bx^\ast)^{-1})\| \leq \|dG(\bx^\ast)\|\|H_0 - dG(\bx^\ast)^{-1}\| \leq \sigma (2  \delta \eta)$ and Inequality \eqref{eq:ineq1},
	\begin{align*}
		\|\bx_1 - \bx^\ast\| &\leq \|H_0\|K\epsilon^d\|\bx_0 - \bx^\ast\| + 2\sigma \delta \eta \|\bx_0 - \bx^\ast\|
		&= \left[\|H_0\|K \epsilon^d + 2 \sigma \epsilon \delta \eta\right]\|\bx_0 - \bx^\ast\|\,.
	\end{align*}
	But $\|H_0\| \leq \|H_0 - dG(\bx^\ast)^{-1}\| + \|dG(\bx^\ast)^{-1}\| \leq 2 \eta \delta + \gamma$. Therefore,
	\[
	\|\bx_1 - \bx^\ast\| \leq [(2 \eta \delta + \gamma)K \epsilon^d + 2 \sigma \epsilon \delta \eta]\|\bx_0 - \bx^\ast\| \leq r\|\bx_0 - \bx^\ast\|
	\]
	using the inequality in Eq.\eqref{eq:epsilon_delta2}, and hence $\bx_1 \in D$. The rest of the proof proceeds by induction. Assume that $\|H_k - dG(\bx^\ast)^{-1}\|_M < 2\delta$,  which implies  $ H_k \in N_2$ and $\|\bx_{k+1} - \bx^\ast\| \leq r\|\bx_k - \bx^\ast\|$ for $k = 0, 1, \dots, m-1$. Now since $\bx_k \in N_1 \subseteq S$, we have $\|F(\bx_k) - \bx^\ast\|^d \leq \tau^d \|\bx_k - \bx^\ast\|^d$ by the local convergence of the MM algorithm in $S$. It follows from the inequality in Eq.\eqref{eq:jacobian_error} that
	\begin{align*}
		\|H_{k+1} - dG(\bx^\ast)^{-1}\|_M - \|H_k - dG(\bx^\ast)^{-1}\|_M &\leq \alpha_1\|\bx_k - \bx^\ast\|^d\|H_k - dG(\bx^\ast)^{-1}\|_M + \alpha_2 \|\bx_k - \bx^\ast\|^d \\
		&\leq \alpha_1 (r^{kd}\epsilon^d)(2\delta) + \alpha_2 r^{kd}\epsilon^d\,.
	\end{align*}
	Therefore, from the inequality in Eq.\eqref{eq:epsilon_delta1}, we have
	\[
	\|H_m - dG(\bx^\ast)^{-1}\|_M \leq \|H_0 - dG(\bx^\ast)^{-1}\|_M + (2 \alpha_1 \delta + \alpha_2)\dfrac{\epsilon^d}{1 - r^d} \leq 2\delta\,.
	\]
	In this way, the induction step is completed by following the same proof as the case for $m=1$. In particular, since $\|H_m - dG(\bx^\ast)^{-1}\| \leq 2\eta \delta$, the Banach lemma implies that
	\[
	\|H_m^{-1}\| \leq \dfrac{\sigma}{1-r}\,.
	\]
	
\end{proof}

\subsection{Proof of Theorem 2}

In order to show that our algorithm satisfies \eqref{eq:jacobian_error} for meeting the conditions of Theorem~\ref{th:convergence}, we write \eqref{eq:BFGS_update} as
\begin{equation} \label{eq:MEM}
	\bar{E} = E\left[I_p - \dfrac{M^{-1}\bv(M\bv)^T}{\|\bv\|^2}\right] + \dfrac{M(\bu - dG(\bx^\ast)^{-1}v)(M\bv)^T}{\|\bv\|^2},
\end{equation}
where $\bar{E} = M(\bar{H} - dG(\bx^\ast)^{-1})M \text{ and } E = M(H - dG(\bx^\ast)^{-1})M$. Eq.\eqref{eq:MEM} will allow us to derive the relationship between $\|H-dG(\bx^\ast)^{-1}\|_M$ and $\|\bar{H}-dG(\bx^\ast)^{-1}\|_M$ satisfying the inequality in \eqref{eq:jacobian_error}. For this purpose, we present the following technical lemma with four important inequalities satisfied by our algorithm. Since our update formulation falls in the classical line of thought, we inherit the following properties directly from the analysis presented by \cite{broyden1973local}, so the proofs have been omitted. Nonetheless, the results have been included for the sake of completion for Theorem~\ref{th:qnm_convergence}.
\begin{lemma} \label{lemma:MEM}
	
	Let $M \in \mathbb{R}^{p\times p}$ be a non-singular symmetric matrix such that
	\begin{equation} \label{eq:lemma2_condition}
		\|M\bc - M^{-1}\bd\| \leq \beta \|M^{-1}\bd\|
	\end{equation}
	for some $\beta \in [0, 1/3]$ and vectors $\bc$ and $\bd$ in $\mathbb{R}^p$ with $\bd \neq \mathbf{0}$. Then using $E$ and $\bar{E}$ as defined earlier,
	\begin{enumerate}
		\item $(1 - \beta)\|M^{-1}\bd\|^2 \leq \bc^T\bd \leq (1 + \beta) \|M^{-1}\bd\|^2$\,.
		\item $E\left[I - \dfrac{(M^{-1}\bd(M^{-1}\bd)^T}{\bc^T\bd}\right] \leq \sqrt{1 - \alpha \theta^2} \|E\|_F$\,.
		\item $\left\|E \left[I - \dfrac{M^{-1}\bd (M\bc)^T}{\bc^T\bd}\right]\right\|_F \leq \left[ \sqrt{1 - \alpha \theta^2} + (1-\beta)^{-1} \dfrac{\|M\bc - M^{-1}\bd\|}{\|M^{-1}\bd\|} \right]\|E\|_F $\,, \\
		where
		\[
		\alpha = \dfrac{1 - 2\beta}{1  - \beta^2} \in [3/8, 1]
		\]
		and 
		\[
		\theta  = \dfrac{\|EM^{-1}\bd\|}{\|E\|_F\|M^{-1}\bd\|} \in [0,1]\,.
		\]
		Moreover, for any $\ba \in \mathbb{R}^p$,
		
		\item $\left\| \dfrac{(\ba - dG(\bx^\ast)^{-1}\bd)(M\bc)^T}{\bc^T\bd} \right\|_F \leq 2\dfrac{\|\ba - dG(\bx^\ast)^{-1}\bd\|}{\|M^{-1}\bd\|}$\,.
	\end{enumerate}
\end{lemma}
In the following proof, we will use results from the above lemma with $\bc = \bv$, $\bd= \bv$, and result $(d)$ particularly for $\ba = \bu$. We are now ready to show that the conditions of Theorem~\ref{th:convergence} are satisfied by our update formula. This allows us to construct the exact neighborhood for each $r \in (0,1)$ wherein our algorithm converges to the stationary point with rate $r$.

\begin{proof}[Theorem~\ref{th:qnm_convergence}]
	Firstly, we construct the neighborhoods $N_1$ and $N_2$ wherein our the updates in \eqref{eq:QN_update} and \eqref{eq:BFGS_update} are well-defined. Define $N_2 = \{H \in \mathbb{R}^{p\times p}: \|dG(\bx^\ast)\|\|H - dG(\bx^\ast)^{-1}\| < 1/2\}$ such that each $H \in N_2$ is non-singular, and there exists a constant $\nu > 0$ such that $\|H\| \leq \nu$ for all $H \in N_2$. Using Lemma~\ref{lemma:lipchitz}, we we also choose $\epsilon \text{ and } \rho$ such that $\max\{\|\bar{\bx} - \bx^\ast\|^d, \|\bx-\bx^\ast\|^d\} \leq \epsilon$ implies that \eqref{eq:ineq2} holds. In particular, if $\|\bx - \bx^\ast\| \leq \epsilon$ and $H \in N_2$, then $\bx \in D$ and 
	\[
	(1/\rho)\|\bx - \bx^\ast\| \leq \|G(\bx)\| \leq \rho \|\bx - \bx^\ast\|\,.
	\]
	As a consequence of applying the inequality above on Eq.\eqref{eq:QN_update}, we have
	\[
	\|\bs\| \leq \|H\|\|G(\bx)\| \leq \rho \|H\| \|\bx - \bx^\ast\| \leq \rho \nu \|\bx- \bx^\ast\|\,.
	\]
	Now define $N_1$ as the set of all $\bx\in \mathbb{R}^p$ such that 
	\[
	\|\bx - \bx^\ast\| < \min\{\epsilon/2, \epsilon/(2\rho), \epsilon/(2\rho \nu)\}
	\]
	where $\rho \epsilon < (3\mu_2)^{-1/d}$. Now if $\bx \in N_1$ then
	\begin{align} \label{eq:ineq4}
		\|\bs\| &\leq \rho \nu \|\bx-\bx^\ast\| \leq \epsilon/2 \qquad \text{ and } \qquad
		\|F(\bx) - \bx\| \leq \rho \|\bx - \bx^\ast\| \leq \epsilon/2\,.
	\end{align}
	If $N = ((N_1 \cap S) \times N_2) \cap N^\prime$ and $(\bx, H) \in N$, then $\bx, \bar{\bx}, \text{ and } F(\bx)$ lie in $D$ because using \eqref{eq:ineq4}
	\[
	\|\bar{\bx} - \bx^\ast\| \leq \|\bs\| + \|\bx - \bx^\ast\| \leq \epsilon \quad \text{and} \quad \|\bar{F(\bx)} - \bx^\ast\| \leq \|F(\bx) - \bx\| + \|\bx - \bx^\ast\| \leq \epsilon\,.
	\]
	Hence Inequality \eqref{eq:ineq2} shows that
	\begin{subequations} 
		\begin{align}
			\left(\dfrac{1}{\rho}\right)\|\bs\| &\leq \|\by\| \leq \rho \|\bs\|\,,\\
			\left(\dfrac{1}{\rho}\right)\|\bu\| &\leq \|\bv\| \leq \rho \|\bu\|\,, \label{eq:uv_ineq}
		\end{align}
	\end{subequations}
	
	and in particular
	\begin{subequations}
		\begin{align}
			\mu_2 \|\by\|^d &\leq \mu_2 (\rho \epsilon)^d \leq 1/3\\
			\mu_2 \|\bv\|^d &\leq \mu_2 (\rho \epsilon)^d \leq 1/3\,. \label{eq:beta_value}
		\end{align}
	\end{subequations}
	Thus $\by=0$ if and only if $\bs=0$, which happens if and only if $\bx=\bx^\ast$. This shows that the update function in Eq.\eqref{eq:QN_update} and \eqref{eq:BFGS_update} is well-defined for all $(\bx,H) \in N$. We now show that the update functions satisfy the conditions of Theorem \ref{th:convergence}. Since \eqref{eq:ineq3} and \eqref{eq:beta_value} imply that \eqref{eq:lemma2_condition} hold with $\beta = 1/3$, it follows from \eqref{eq:ineq3} and parts $(c)$, $(d)$ of Lemma \ref{lemma:MEM} applied on \eqref{eq:MEM} that
	\begin{align*}
		\|\bar{H} - dG(\bx^\ast)^{-1}\|_M &\leq \left[\sqrt{1 - \dfrac{3}{8}\theta^2} + \dfrac{3}{2}\mu_2\|\bv\|^d\right]\|H - dG(\bx^\ast)^{-1}\|_M\\
		&\qquad + \dfrac{2\|M\| \|\bu - dG(\bx^\ast)^{-1}\bv\|}{\|M^{-1}\bv\|}
	\end{align*}
	where $\theta = \dfrac{\|M[H - dG(\bx^\ast)^{-1}]\bv\|}{\|H- dG(\bx^\ast)^{-1}\|_M\|M^{-1}\bv\|}$.
	But for any $\bx\in N_1$,
	\[
	\|\bu - dG(\bx^\ast)^{-1}\bv\| \leq K\|dG(\bx^\ast)^{-1}\|\max\{\|F(\bx) - \bx^\ast\|^d, \|\bx-\bx^\ast\|^d\}\|\bu\|\,.
	\]
	Using Inequality \eqref{eq:uv_ineq},
	\begin{align} \label{eq:H_ineq}
		\|\bar{H} - dG(\bx^\ast)^{-1}\|_M &\leq \sqrt{1 - \dfrac{3}{8}\theta^2}\|H - dG(\bx^\ast)^{-1}\|_M \nonumber \\
		& \qquad + \max\{\|F(\bx) - \bx^\ast\|^d, \|\bx-\bx^\ast\|^d\}[\alpha_1\|H - dG(\bx^\ast)^{-1}\|_M + \alpha_2]
	\end{align}
	where $\alpha_1 = \left(\dfrac{2}{3}\right)(2 \rho)^d \mu_2$ and $\alpha_2 = 2 \rho K \|M\|^2 \|dG(\bx^\ast)^{-1}\|$. This inequality now satisfies the assumptions of Theorem \ref{th:convergence} and therefore, $\bx_k$ converges locally to $\bx^\ast$ as $k$ increases.

\end{proof}

\begin{proof}[Proof of Corollary \ref{cor:q-superlinear}]
	
	To prove the Q-superlinear convergence, we use the Corollary~\ref{cor:superlinear_conv} that guarantees the desired result if a subsequence of $\{H_k\}$ converges to $dG(\bx^\ast)^{-1}$. Define
	\[
	\theta_k = \dfrac{\|M[H_k - dG(\bx^\ast)^{-1}]\bv_k\|}{\|H_k - dG(\bx^\ast)^{-1}\|_M\|M^{-1}\bv_k\|}\,.
	\]
	Since $\sqrt{1 - \alpha} < 1 - \alpha/2$, from Eq.\eqref{eq:H_ineq} we get
	\begin{align} \label{eq:H_error_bound} 
		\dfrac{3}{16}\theta_k^2\|H - dG(\bx^\ast)^{-1}\|_M &\leq \left[\|H - dG(\bx^\ast)^{-1}\|_M - \|\bar{H} - dG(\bx^\ast)^{-1}\|_M\right] \nonumber \\
		& \qquad + \max\left\{\|F(\bx) - \bx^\ast\|^d, \|\bx-\bx^\ast\|^d\right\}[\alpha_1\|H - dG(\bx^\ast)^{-1}\|_M + \alpha_2]\,.
	\end{align}
	If there is a subsequence $\{H_k\}$ such that it converges to $dG(\bx^\ast)^{-1}$, then $\|H - dG(\bx^\ast)^{-1}\|_M$ converges to zero and we are done by Corollary~\ref{cor:superlinear_conv}. Otherwise, $\|H - dG(\bx^\ast)^{-1}\|_M$ is bounded by $\alpha$ but does not converge to zero. 
	Now summation on Eq.\eqref{eq:H_error_bound} yields
	\begin{align*}
		\dfrac{3}{16}\sum_{k=1}^{\infty}\theta_k^2\|H_k - dG(\bx^\ast)^{-1}\|_M &\leq \|H_0 - dG(\bx^\ast)^{-1}\|_M  - \|H_{\infty} - dG(\bx^\ast)^{-1}\|_M\\
		&+  [\alpha_1 \alpha + \alpha_2]\epsilon^d \sum_{k=1}^{\infty} r^{(k-1)d}\\
		&\leq 2\alpha + \dfrac{[\alpha_1 \alpha + \alpha_2]\epsilon^d}{1 - r^d} < \infty\,.
	\end{align*}
	and since
	\begin{align*}
		\sum_{k=1}^{\infty}\theta_k^2\|H_k - dG(\bx^\ast)^{-1}\|_M &= \sum_{k=1}^{\infty}\dfrac{\|M[H_k - dG(\bx^\ast)^{-1}]\bv_k\|^2}{\|H_k - dG(\bx^\ast)^{-1}\|_M\|M^{-1}\bv_k\|^2}\\
		&\geq \dfrac{1}{\alpha} \sum_{k=1}^{\infty}\dfrac{\|M[H_k - dG(\bx^\ast)^{-1}]\bv_k\|^2}{\|M^{-1}\bv_k\|^2}\,,
	\end{align*}
	this forces the following limit to converge to zero. 
	\begin{equation} \label{eq:lim_(H-G'y)/(y)}
		\lim_{k \to \infty}\dfrac{\|[H_k - dG(\bx^\ast)^{-1}]\bv_k\|}{\|\bv_k\|} = 0\,.
	\end{equation}
	Using $H_k \bv_k = H_k G(F(\bx_k)) - H_k G(\bx_k) = H_k G(F(\bx_k)) + \bs_k$, we can write
	\[
	[H_k - dG(\bx^\ast)^{-1}]\bv_k = H_k G(F(\bx_{k})) - dG(\bx^\ast)^{-1}[\bv_k - dG(\bx^\ast)s_k]\,.
	\]
	At the end of the proof of Theorem~\ref{th:convergence}, we prove that there exists $\upsilon > 0$ such that $\|H_k\| \leq \upsilon$. Using the above equation, Lemma~\ref{lemma:lipchitz}, and the fact that $\|\bx_{k+1} - \bx^\ast\| \leq \|\bx_k - \bx^\ast\|$, we get
	
	\begin{align} \label{eq:G(F(x))_upperbound}
		\|G(F(\bx_{k}))\| &\leq \|H_k^{-1}\|\left[ \|[H_k - dG(\bx^\ast)^{-1}]\bv_k\| + \|dG(\bx^\ast)^{-1}\| \|\bv_k - dG(\bx^\ast)\bs_k\| \right] \nonumber \\
		& \leq \upsilon \big[ \|[H_k - dG(\bx^\ast)^{-1}]\bv_k\| +  \|dG(\bx^\ast)^{-1}\| K \|\bx_k - \bx^\ast\|^d \|\bu_k\| \nonumber 
		\\ & \qquad + \|dG(\bx^\ast)^{-1}\| \|G^{\prime} (\bx^\ast) (\bs_k - \bu_k)\| \big] \nonumber 
		\\ &\leq \upsilon \big[ \|[H_k - dG(\bx^\ast)^{-1}]\bv_k\| + \|dG(\bx^\ast)^{-1}\| K \|\bx_k - \bx^\ast\|^d \|\bu_k\| \nonumber
		\\  & \qquad + \|dG(\bx^\ast)^{-1}\| \|dG(\bx^\ast)\| \|(\bx_{k+1} - F(\bx_k) )\| \big] \, .
	\end{align}

	A critical assumption is the condition that $\lim_{k \to \infty} \|\bx_{k+1} - F(\bx_k)\|/\|\bx_k - \bx^\ast\| = 0$. Since $\|\bu_k\| \geq (1/\rho)\|\bv_k\|$ and appealing to the limit in Eq.\eqref{eq:lim_(H-G'y)/(y)},
	\
	\[
	\lim_{k \to \infty} \dfrac{\|G(F(\bx_{k}))\|}{\|\bu_k\|} = 0 .
	\]
	Now using Eq.\eqref{eq:ineq2}, we know that $\|F(\bx_k) - \bx^\ast\| \leq \rho\| G(F(\bx_k))\|$ and $\|\bx_k - \bx^\ast\| \geq \|\bu_k\|/\rho$. Therefore we have the string of inequalities
	\[
	\dfrac{\|\bx_{k+1} - \bx^\ast\|}{\|\bx_{k} - \bx^\ast\|} \leq \dfrac{\|\bx_{k+1} - F(\bx_k)\|}{\|\bx_{k} - \bx^\ast\|} + \dfrac{\|F(\bx_k) - \bx^\ast\|}{\|\bx_{k} - \bx^\ast\|} \leq \dfrac{\|\bx_{k+1} - F(\bx_k)\|}{\|\bx_{k} - \bx^\ast\|} + \dfrac{\rho^2\| G(F(\bx_k))\|}{\|\bu_k\|}\,.
	\]
	Q-superlinearity follows as the upperbound of ${\|\bx_{k+1} - \bx^\ast\|}/{\|\bx_{k} - \bx^\ast\|}$ goes to $0$ as $k \to \infty$. This concludes the proof of Q-superlinearity of our quasi-Newton method for MM acceleration in a neighborhood of the limit point.
\end{proof}

\subsection{Examples}

\subsubsection{Truncated Beta Binomial}

As discussed earlier in Example~\ref{ex:trunc.beta.binom}, we have the \cite{lidwell1951observations} dataset of cold incidences in households of size four. These holuseholds are classified as: (a) adults only, (b) adults and school children, (c) adults and infants, and (d) adults, school children, and infants. Only households with at least one cold incidence are reported, hence warranting the use of zero-truncated beta-binomial distribution to model the dataset.

We have already presented a comparative analysis of different MM acceleration methods for dataset subcategory (a) in Table~\ref{tab:beta_binom1}. Using the same tolerance $\epsilon$ and starting points $(\pi_0, \alpha_0)$, we run the methods for the other three subcategories and present the unified results in Table~\ref{tab:beta_binom2}. It is worth to note that while at least one SQUAREM methods fail to provide acceleration for each subcategory, BQN consistently accelerates over the slow MM algorithm. 

\begin{table}
	
	\centering
	\begin{tabular}{|c| c c c c| c c|} 
	\hline
	Household & \multicolumn{4}{c|}{Number of cases} & \multicolumn{2}{c|}{MLE} \\ [0.5ex] 
	& 1 & 2 & 3 & 4 & $\hat{\pi}$ & $\hat{\alpha}$ \\ [0.5ex]
	\hline
	(a) & 15 & 5 & 2 & 2 &  0.0000 & 0.6151 \\ 
	(b) & 12 & 6 & 7 & 6 & 0.1479 & 1.1593 \\
	(c) & 10 & 9 & 2 & 7 & 0.0000 & 1.6499 \\
	(d) & 26 & 15 & 3 & 9& 0.0001 & 1.0594 \\[1ex] 	
	\hline
	\end{tabular}
	\caption{The Lidwell and Somerville (1951) cold data on households of size $4$ and corresponding MLEs 	under the truncated beta-binomial model.}	
	\label{tab:bet_binom_data}
\end{table}

\begin{table}
	\caption{\label{tab:beta_binom2}Truncated beta binomial: comparison of algorithms for the Lidwell and Somerville
		Data. The starting point is $(\pi,\alpha) = (0.5, 1)$, the stopping criterion
		is $\epsilon = 10^{-7}$, and the number of parameters is two.}
	\centering
	\fbox{\begin{tabular}{c | c c c c c}
			Data & Algorithm & -ln L & Fevals & Iterations & Time (in sec) \\ [0.5ex] 
			\hline
			
			(b) & MM & 41.7286 & 5492 & 5492 & 0.026 \\
			& BQN, $q=1$ & 41.7286 & 1012 & 507 & 0.036\\
			& SqS1 & 41.7286 & 248 & 210 & 0.019\\
			& SqS2 &41.7286 & 1553 & 1148 & 0.106\\
			& SqS3 &41.7286 & 79 & 40 & 0.006\\
			& ZAL, $q=2$ & 41.7286 & 1136 & 1132 & 0.063\\ [1ex]
			(c) & MM & 37.3586 & 61843 &  61843 & 0.323\\
			& BQN, $q=1$ & 37.3589 & 1864 & 933 & 0.062\\
			& SqS1   & 37.3587 & 1370 & 1345 & 0.122 \\
			& SqS2 & 37.3582 & 9649 & 8966 & 0.977 \\
			& SqS3 & 37.3582 & 145 & 73 & 0.010\\
			& ZAL, $q=2$ & 37.3582 & 28 & 24 & 0.004\\[1ex]
			(d) & MM & 65.0423 & 25026 & 25026 & 0.132     \\
			& BQN, $q=1$ & 65.0435 & 268 & 135 &  0.007 \\
			& SqS1 & 65.0413 & 1648 & 1622 & 0.136 \\
			& SqS2 & 65.0420 & 5727 & 5443 & 0.472 \\
			& SqS3 & 65.0402 & 97 & 49 & 0.007 \\
			& ZAL, $q=2$ & 65.0402 & 25 & 21 & 0.003 \\ [1ex]
		
		\end{tabular}
	}
	
\end{table}

\end{document}